\documentclass[]{article}
\usepackage[utf8]{inputenc}

\usepackage[dvipsnames]{xcolor}
\definecolor{myred}{HTML}{c20014}
\definecolor{mygreen}{HTML}{008000}

\usepackage{amsfonts,latexsym,amsmath,amssymb}
\usepackage{amsthm}
\usepackage{dsfont}
\usepackage{enumitem}
\usepackage{cite}

\usepackage{url}
\usepackage[colorlinks=true,linkcolor=myred,citecolor=mygreen]{hyperref}
\usepackage[capitalise]{cleveref}
\crefname{equation}{}{}
\Crefname{equation}{Equation}{}
\usepackage{geometry}
\newtheorem{theorem}{Theorem}[section]
\newtheorem{lemma}[theorem]{Lemma}

\theoremstyle{remark}
\newtheorem{rem}[theorem]{Remark}
\newtheorem{hyp}[theorem]{Hypothesis}

\crefname{hyp}{Hypothesis}{Hypotheses}

\numberwithin{equation}{section}

\newcommand{\M}{\mathcal{M}}
\newcommand{\dt}[1]{\frac{\mathrm{d}#1}{\mathrm{d}t}}
\newcommand{\D}{\mathcal{D}}
\newcommand{\A}{\mathcal{A}}

\renewcommand{\d}{\mathrm{d}}

\newcommand{\yr}{y_\mathrm{ref}}
\newcommand{\Acl}{\tilde{\A}_{d, \yr}}

\newcommand{\K}{\mathcal{K}}

\renewcommand{\epsilon}{\varepsilon}
\newcommand{\eps}{\varepsilon}

\renewcommand{\S}{\mathcal{S}}

\newcommand{\R}{\mathbb{R}}
\newcommand{\T}{\mathcal{T}}

\renewcommand{\L}{\mathcal{L}}
\renewcommand{\leq}{\leqslant}
\renewcommand{\geq}{\geqslant}

\title{Output regulation of infinite-dimensional nonlinear systems: a forwarding approach for contraction semigroups\thanks{
\textbf{Funding:} {This work was partially funded by the French Grants ANR ODISSE (ANR-19-CE48-0004-01) and MIAI@Grenoble Alpes (ANR-19-P3IA-0003).}
}}
\author{Nicolas Vanspranghe\thanks{Mathematics and Statistics, Faculty of Information Technology and Communication Sciences, Tampere University, P.O. Box 692, 33101 Tampere, Finland. Email: \nolinkurl{nicolas.vanspranghe@tuni.fi}.}
\and
Lucas Brivadis\thanks{Université Paris-Saclay, CNRS, CentraleSupélec, Laboratoire des Signaux et Systèmes, 91190, Gif-sur-Yvette, France. Email:
\nolinkurl{lucas.brivadis@centralesupelec.fr}.}
}
\date{}

\begin{document}

\maketitle

\begin{abstract}
This paper deals with the problem of robust output regulation of systems governed by nonlinear contraction semigroups.
After adding an integral action to the system,
we design a feedback law based on the so-called forwarding approach.
For small constant perturbations, we give sufficient conditions for
the existence of a locally exponentially stable equilibrium at which the output coincides with the reference.
Under additional assumptions, global asymptotic stability is achieved.
All these conditions are investigated in the case of semilinear systems, and examples of application are given.
\end{abstract}



\section{Introduction}

Robust output regulation in one of the oldest problems in control theory.
It consists in designing a feedback law which ensures that the output of a system tracks a given reference, even in the presence of external disturbances.
If a feedback law can be designed to stabilize the system at some target point
then a natural strategy to regulate the output  robustly with respect to constant disturbances consists in adding an integral action to the system.
Then, one must find a new feedback law that stabilizes the augmented system (i.e., the state and the output integrator) at some equilibrium point even in the presence of constant disturbances.
The integrator guarantees that the output is at the reference when the full state is at the equilibrium.

While the theory of output regulation for linear finite-dimensional systems is well-understood since the seminal work of E. Davidson \cite{Davison, 1101104} (see also \cite{francis1975internal}), many problems remain open in the case of nonlinear and/or infinite-dimensional systems.
In \cite{Pohjolainen, POHJOLAINEN1985622}, S. Pohjolainen proposed an extension
of the finite-dimensional linear robust regulation theory to the infinite-dimensional context. In particular, a controller based on an output integrator was investigated.
More recently, this work has been continued in \cite{bymes2000output,rebarber2003internal,paunonen2010internal,paunonen2015controller, doi:10.1137/17M1136407} where an internal model principle was developed to reject a wider class of disturbances (and not only constant ones). 
Other works
focusing on output regulation of particular linear partial differential equations include \cite{guo2018adaptive,lhachemi2021,guo2022output}, to mention only a few.
In comparison, very few attempts to tackle infinite-dimensional nonlinear output regulation problems exist in the literature, especially in an abstract framework. For instance, \cite{logemann2000time} dealt with single-input single-output linear distributed parameter systems subject to some input nonlinearity, while \cite{natarajan2016approximate} focused on local approximate regulation of Lipschitz perturbations of linear systems. 

In the finite-dimensional context, a strategy to take nonlinearities into account is the so-called \emph{forwarding} approach, devised by F. Mazenc and L. Praly in \cite{praly-mazenc-forwarding} and developed with different purposes in \cite{kaliora2005stabilization,benachour2013forwarding, astolfi-praly, poulain2010robust}.
We propose to approach the issue of constant output regulation via the forwarding technique. Few extensions of this methodology in the infinite-dimensional setting exist.
Control systems investigated in \cite{terrand2019adding,marx:hal-03417238} are linear. The article \cite{marx:hal-02944073} dealt with cone-bounded nonlinearities (such as saturation) applied to the input.
More recently, \cite{balogoun:hal-03292801} dealt with the regulation of the linearized Korteweg-de Vries equation, and local regulation of the nonlinear model with the same feedback law was proved.
In the recent work \cite{mattia-forwarding}, the authors followed a contraction approach to achieve the output regulation of nonlinear finite-dimensional systems by means of an incremental forwarding approach.
Inspired by this work, we focus on systems of contraction, namely, such that the distance between any pair of trajectories is exponentially converging towards zero. While we do not restrict ourselves to Lipschitz perturbations of linear systems (which we will refer to as \emph{semilinear} systems), we pay special attention to this case.

\paragraph{Organization of the paper}

In the first section of the paper, we state more precisely the class of systems under consideration and the output regulation problem.
Then, we give some insights on the forwarding strategy we use to design a stabilizing feedback law.
Main results are stated and discussed in \cref{sec:main-res}, and proved in \cref{sec:proof}.
Some examples of application are provided in \cref{sec:examples}.

\paragraph{Notation}
If $H$ is a Hilbert space endowed with the norm $\|\cdot\|_H$,
$\mathcal{B}_{H}(x, r)$ denotes the open ball of $(H, \|\cdot\|_H)$ centered at $x\in H$ of radius $r>0$.
Let $Z$ be a Hilbert space.
We denote by
$\mathcal{C}(H, Z)$ the space of continuous maps from $H$ to $Z$ and by
$\mathcal{C}^1(H, Z)$ the space of continuously Fréchet differentiable maps. We also denote by $\mathcal{L}(H, Z)$ the Banach space of bounded linear operators from $H$ to $Z$. Given $L$ in $\mathcal{L}(H, Z)$, the adjoint operator $L^*$ is uniquely defined in $\mathcal{L}(Z, H)$ by $(Lx, y)_Z = (x, L^*y)_H$, where $(\cdot, \cdot)$ denotes the scalar product.
We shall say that a map is locally Lipschitz continuous if it is Lipschitz continuous on every bounded set. Given a $H$-valued map $\mathcal{A}$ defined on some subset $\D(\A)$ of $H$, we say that $w$ is a \emph{strong} solution to the abstract evolution equation $\d w / \d t + \A(w) = 0$ on $[0, T]$ if $w$ is absolutely continuous with respect to the topology of $H$, takes values in $\D(\A)$, and solves the differential equation almost everywhere (a.e.) on $(0,T)$.
A \emph{weak} solution $w$ is a limit of such strong solutions in $\mathcal{C}([0, T], H)$.

\section{Problem statement}  
\subsection{Systems under consideration}
Consider a nonlinear infinite-dimensional controlled system with measured output of the form:
\begin{subequations}\label{eq:open-loop}
\begin{align}
&\dt{w} + \A(w) = Bu(t),\\
&y = Cw,
\end{align}
\end{subequations}
where $w$ is the state of the system lying in some real
Hilbert space $H$,
$y$ is the measured output lying in some real Hilbert space $Z$,
$u$ is the control input lying in some real Hilbert space $U$,
$\A:\D(\A)\to H$ is a singled-valued maximal monotone operator defined on the dense subset $\D(\A)\subset H$ such that $\A(0) = 0$,
$B:U\to H$ is a bounded linear map,
and $C:\D(C)\to Z$ is a (potentially unbounded) linear map defined on the (dense) subspace $\D(C)\supset\D(\A)$.
According to \cite[Chapter 4, Proposition 3.1]{showalter_monotone_2013},
$-\A$ is the generator of a nonlinear strongly continuous contraction semigroup over $H$, denoted
by $\{\T_t\}$.
In order to address the problem of output regulation, we assume that $\A$ satisfies the following assumption.
\begin{hyp}[Monotonicity]\label{hyp:nonlin-coer}
The nonlinear operator $\A$ is strongly monotone, {i.e.}, there exists $\alpha>0$ such that
\begin{equation}
\label{eq:nonlin-coer}
(\A(w_1) - \A(w_2), w_1 - w_2)_H \geq \alpha \|w_1 - w_2\|^2_H, \quad \forall w_1, w_2 \in \D(\A).
\end{equation}
\end{hyp}
\cref{hyp:nonlin-coer} implies that the
semigroup $\{\T_t\}$ satisfies
\begin{equation}
\label{eq:exp-cont}
\|\T_tw_1 - \T_tw_2\|_H \leq \exp(-\alpha t) \|w_1 - w_2\|_H, \quad \forall w_1, w_2 \in H, \ \forall t \geq 0.
\end{equation}
We address the problem of regulating the output $y$ of \eqref{eq:open-loop} at some constant reference $\yr\in Z$ in presence of constant disturbances $d\in H$ acting on the dynamics of the system. We proceed by adding an integrator to the output of the system, so that the resulting full system may be written as follows:
\begin{subequations}
\label{eq:control-system}
\begin{align}
    &\dt{w} + \A(w) = Bu(t) + d, \\
    &\dt{z} = Cw - y_\mathrm{ref},
    \end{align}
\end{subequations}
where $z$ lying in $Z$ is the output integrator.
The reason for that strategy is the following.
Assume that a feedback law $\varphi: H\times Z\to U$ may be designed, so that \eqref{eq:control-system} in closed-loop with $u(t) = \varphi(w(t), z(t))$ is asymptotically stable at $(0, 0)\in H\times Z$ when $[d, \yr] = (0, 0)$.
Then, formally, one can expect that for small perturbations of this system, {i.e.}, for small values of $[d, \yr]$,
the closed-loop is still asymptotically stable at some new equilibrium point $[w^\star, z^\star]$ near $(0, 0)\in H\times Z$.
Roughly speaking, this would imply that the dynamics of $z$, namely $\d z / \d t = Cw-\yr$, still tend towards zero, that is, the output $y=Cw$ is regulated at $\yr$.
In this paper, we propose a design strategy of such a feedback law $\varphi$ based on the so-called forwarding approach, for which we make precise this formal reasoning.

\subsection{Forwarding design}
\label{sec:forwarding}

Following the forwarding design \cite{praly-mazenc-forwarding}, we consider a candidate Lyapunov function of the form:
\begin{equation}\label{eq:lyapunov}
    V(w, z) = \frac{1}{2}\|w\|^2_H + \frac{\rho}{2}\|z - \M(w)\|^2_Z
\end{equation}
where $\|w\|^2_H / 2$ corresponds to the Lyapunov function associated to the contraction semigroup $\{\T_t\}$,
$\M:H\to Z$ is some Fréchet differentiable nonlinear map that vanishes at zero to be tuned, and
$\rho$ is some positive constant to be fixed (large enough) later.
Let us compute the derivative of $V$ along a strong solution $[w, z]$ of the open-loop system \eqref{eq:control-system}
in the case where $[d, \yr] = [0, 0]$:
\begin{equation}
\begin{aligned}
\dt{} V(w, z)
&= - (w, \A(w))_H + (w, Bu(t))_H\\
&\qquad+\rho( z - \M(w), \d\M(w)\A(w) + Cw -\d\M(w)Bu(t)))_Z.
\end{aligned}
\end{equation}
Under \cref{hyp:nonlin-coer}, we have
$(w, \A(w))_H \geq \alpha \|w\|^2_H$.
The key idea of the forwarding design is to choose a nonlinear map $\M$ satisfying the following functional equation, which guarantees that the graph of $\M$ is positively invariant under the uncontrolled dynamics \cref{eq:control-system}:
\begin{subequations}
\label{eq:forwarding-func}
\begin{align}
 &\M(0) = 0, \\
 &\d\M(w)\A(w) + Cw = 0, \quad \forall w \in \D(\A).
\end{align}
\end{subequations}
In order to follow the forwarding approach, let us consider the next hypothesis.
\begin{hyp}\label{hyp:forwarding-func}
The functional equation \eqref{eq:forwarding-func} admits a continuously Fréchet differentiable solution
$\M \in \mathcal{C}^1(H, Z)$.
\end{hyp}
For such a map $\M$, we obtain:
\begin{equation}
\label{eq:pre-lyap}
\dt{} V(w, z)
\leq - \alpha \|w\|^2_H + (B^*w, u(t))_U
-\rho(B^*\d\M(w)^*[z - \M(w)], u(t))_U.
\end{equation}
This suggests to set the feedback law as
\begin{equation}
\label{eq:feedback-law}
    u(t) = B^*\d\M(w)^*[z - \M(w)].
\end{equation}
Then, applying Young's inequality, we get that for all $\eps>0$,
\begin{equation}
\label{eq:lyap}
\dt{} V(w, z)
\leq - \alpha \|w\|^2_H + \frac{\eps}{2} \|B\|^2_{\mathcal{L}(U, H)} \|w\|^2_H
+ \frac{1}{2\eps}\|u(t)\|_U^2
-\rho\|u(t)\|_U^2.
\end{equation}
In particular, setting $\rho = {\|B\|^2_{\mathcal{L}(U, H)}}/{\alpha}$
and $\eps = 1/\rho$ yields
\begin{equation}
\dt{} V(w, z)
\leq - \frac{\alpha}{2} \|w\|^2_H
-\frac{\rho}{2}\|u(t)\|_U^2\leq0,
\end{equation}
so that $V$ is indeed a Lyapunov function.
These formal computations give an insight on the feedback law \eqref{eq:feedback-law} to choose and on the Lyapunov function \eqref{eq:lyapunov} to consider in order to stabilize \cref{eq:control-system}.
However, many important questions remain open, in particular concerning the well-posedness of the closed-loop, the stability properties in presence of $[d, \yr]$, and the feasibility of the functional equation \eqref{eq:forwarding-func}.

\begin{rem}
Checking \cref{hyp:forwarding-func} on a general nonlinear system of the form \eqref{eq:open-loop} is difficult in general.
Moreover, since $\M$ is used in feedback law \cref{eq:feedback-law}, an expression of $\M$ in terms of parameters of the system should be given in order to implement the controller.
For these two reasons, we investigate later in the paper the case of semilinear systems, for which we give sufficient conditions on $\A$ for \cref{hyp:forwarding-func} to be satisfied, as well as an expression of $\M$ in terms of $\A$ and $C$.
\end{rem}

\section{Main results}
\label{sec:main-res}
We now state our results on the closed-loop system \cref{eq:control-system}-\eqref{eq:feedback-law}.
\subsection{Well-posedness of the closed-loop dynamics} Our first result concerns the existence and uniqueness of solutions to the closed-loop equations \cref{eq:control-system}-\eqref{eq:feedback-law}.

\begin{theorem}[Well-posedness]
\label{th:well-posed}
Assume \cref{hyp:nonlin-coer,hyp:forwarding-func} hold and $\d\M$ is locally Lipschitz continuous. Then,
\begin{enumerate}[label = (\roman*)]
    \item For all $[d, \yr]\in H\times Z$ and all $[w_0, z_0] \in H \times Z$, the closed-loop system \cref{eq:control-system}-\eqref{eq:feedback-law}
    admits a unique maximal weak solution $[w, z]$ 
    satisfying the initial condition $[w(0), z(0)]=[w_0, z_0]$ and defined on $[0, T_\mathrm{max})$
    for some $T_\mathrm{max}\in(0, +\infty]$;
    \label{wpi}
    \item Moreover, if $w_0 \in \D(\A)$, then $[w, z]$ is a strong solution to \cref{eq:control-system}-\cref{eq:feedback-law};
    \label{wpii}
    \item Furthermore, if $d=0$ or $\d\M$ is globally Lipschitz continuous, then \cref{eq:control-system}-\eqref{eq:feedback-law} is forward complete, {i.e.},
    $T_\mathrm{max} = + \infty$ for all $[w_0, z_0] \in H \times Z$.
    \label{wpiii}
\end{enumerate}
\end{theorem}

The proof of \cref{th:well-posed} is given in \cref{sec:proof-wp} and works around two technical difficulties
(the possible unboundedness of $C$ and the \emph{a priori} non-monotonicity of the system with integral action)
by an appropriate use of \cref{hyp:forwarding-func}.

\subsection{Sufficient conditions for output regulation}
In view of the output regulation problem, we wish to investigate the existence of an attractive equilibrium.
The following theorem is the main result of the paper.

\begin{theorem}[Sufficient conditions for output regulation]
\label{th:global-stab}
Assume \cref{hyp:nonlin-coer,hyp:forwarding-func} are satisfied.
Assume that
\begin{equation}
    \label{eq:range-cond}
\operatorname{Range} \d\M(0)B = Z,
\quad {i.e.},\quad \exists \lambda>0\mid \forall z \in Z,\
    \|B^* \d\M(0)^*z\|^2_U \geq \lambda \|z \|^2_Z.
    \end{equation}
If $\d\M$ is globally (resp. locally) Lipschitz continuous,
then there exist positive constants
$M$, $\kappa$ and $r$ and
a neighborhood $\mathcal{N}$ of the origin in $H\times Z$
such that
for any $[d, y_{\mathrm{ref}}]$ in $\mathcal{B}_{H \times Z}(0, r)$ (resp. in $\{0\}\times\mathcal{B}_{Z}(0, r)$),
the following results hold.
\begin{itemize}
\item 
There exists an equilibrium point $[w^\star, z^\star]\in \mathcal{D}(\mathcal{A}) \times Z$ in $\mathcal{N}$
of the closed-loop system \cref{eq:control-system}-\cref{eq:feedback-law}
such that
$
Cw^\star = y_{\mathrm{ref}}
$ and
for all $[w_0, z_0]$ in $\mathcal{N}$, the corresponding solution $[w, z]$ of \cref{eq:control-system}-\cref{eq:feedback-law} satisfies, for all $t\geq 0$,
\begin{equation}\label{eq:exp-conv}
        \|[w(t) - w^\star, z(t) - z^\star]\|_{H\times Z} \leq
        M \exp(-\kappa t)
         \|[w_0 - w^\star, z_0 - z^\star]\|_{H\times Z}.
    \end{equation}
\item Moreover, if 
    \begin{equation}\label{eq:unif-coercive}
    \|B^* \d\M(w)^*z\|^2_U \geq \lambda \|z \|^2_Z, \quad \forall z \in Z,\ \forall w \in H,
    \end{equation}
then for all $[w_0, z_0]$ in $H\times Z$,  the corresponding solution $[w, z]$ of \cref{eq:control-system}-\cref{eq:feedback-law} satisfies:
\begin{equation}\label{eq:glob-conv}
[w(t), z(t)] \to [w^\star, z^\star] \quad \mbox{in}~H\times Z \quad \mbox{as}~ t \to +\infty.
\end{equation}
\end{itemize}
\end{theorem}
The proof of \cref{th:global-stab} is given in \cref{sec:proof-stab}.
\Cref{eq:exp-conv} guarantees the local exponential stability of the equilibrium point $[w^\star, z^\star]$, while \cref{eq:glob-conv} provides global asymptotic stability under the additional condition \cref{eq:unif-coercive}.
This condition is stronger than a global version of the range condition \cref{eq:range-cond},
in the sense that the constant $\lambda$ must be independent of $w\in H$. We shall show in the proof how this uniformity implies the global attractivity of the local exponential basin of attraction $\mathcal{N}$.

Conditions \cref{eq:range-cond} and \cref{eq:unif-coercive} may be difficult to check in general. In \cref{sec:semilinear}, we show how to investigate them in the semilinear case.
In the finite-dimensional context, \cref{eq:unif-coercive} is equivalent to the incremental condition given in \cite[Assumption 2]{mattia-forwarding}.
In our work, it allows to obtain global asymptotic stability even in the presence of small perturbations.
The local condition \cref{eq:range-cond} is commonly used in the forwarding design (see, e.g., \cite{astolfi-praly}) to obtain local exponential stability. 

It is clear that if the operator $C$ is continuous, {i.e.}, $C$ belongs to $\mathcal{L}(H, Z)$, then the output $Cw(t)$ goes to $\yr$ in $Z$ whenever $w(t)$ goes to $w^\star$ in $H$. If $C$ is unbounded, the question is more delicate and is discussed in \cref{sec:C-unbounded}.

\subsection{The semilinear case}
\label{sec:semilinear}
Consider now the case where $\A = A + F$
with $\D(\A) = \D(A)$, where
$-A$ is the infinitesimal generator of a strongly continuous semigroup $\{\S_t\}$ of linear operators on $H$, and $F$ is a nonlinear mapping satisfying
\begin{equation}\label{hyp:F}
F \in \mathcal{C}^1(H), \quad \d F ~\mbox{locally Lipschitz continuous},
\end{equation}
and without loss of generality, $\d F(0) = 0$.
Following \cite{pazy_semigroups_2012,CurZwa20book}, we shall say that $\A$ is semilinear.
In that context, let us introduce the following set of assumptions.

\begin{hyp}
\label{hyp:semilinear}
The operator $\A = A+F$ is semilinear
and satisfies the following properties:
\begin{enumerate}[label = (\roman*)]
    \item $C$ is $A$-bounded, {i.e.}, there exist positive constants $a$ and $b$ such that
    \begin{equation}
        \|Cw\|_Z \leq a \|Aw\|_H + b \|w\|_H,  \quad \forall w\in\D(A);
    \end{equation}
    \item There exists $\alpha>0$ such that
    \begin{equation}
\label{eq:as-stab-lin}
(Ah + \mathrm{d}F(w)h, h)_H \geq  \alpha \|h\|^2_H, \quad \forall w \in H, \ \forall h \in \mathcal{D}(A).
\end{equation}
\end{enumerate}
\end{hyp}
 
Note that \cref{eq:as-stab-lin} in \cref{hyp:semilinear} implies $\alpha$-exponential stability of the linear semigroup $\{\S_t\}$; in particular, $0$ is in the resolvent set of $A$.
Remark also that \cref{hyp:semilinear} implies both \cref{hyp:nonlin-coer} and maximal monotonicity of $\A$. 
Indeed, writing $F(w_1) - F(w_2)$ as an integral of $\d F$ along the line segment joining $w_1$ to $w_2$, one can show that \cref{eq:as-stab-lin} implies \cref{eq:nonlin-coer}.
Then,
standard results on Lipschitz perturbations of linear systems (see \cite[Section 6.1]{pazy_semigroups_2012} or \cite[Chapter 11]{CurZwa20book}) together with \cref{eq:nonlin-coer} yield that $- \A$ generates a strongly continuous semigroup of contractions (denoted by $\{\T_t\}$) on $H$, which in turn implies that $\A$ is maximal monotone by virtue of \cite[Theorem 4]{komura_differentiability_1969}. 

Under \cref{hyp:semilinear}, we show in the following theorem that \cref{eq:forwarding-func} admits a solution that can be expressed in terms of $C$, $A$, and $F$.
\begin{theorem}[Existence of $\M$ in the semilinear case]
\label{th:M-semilinear}
Assume \cref{hyp:semilinear} is satisfied. Then there exists a map $\M \in \mathcal{C}^1(H, Z)$ satisfying \cref{eq:forwarding-func} with $\d \M$ locally Lipschitz continuous and given for all $w\in H$ by
\begin{equation}
\label{eq:M-int}
\M(w) \triangleq -C \left \{   \lim_{\tau \to + \infty} \int_0^\tau \mathcal{T}_t w \, \mathrm{d}t \right \} = - CA^{-1}w + CA^{-1}\int_0^{+\infty}F(\T_t w) \, \mathrm{d}t.
\end{equation}
For all $w, h\in H$,
\begin{equation}\label{eq:dM}
\d \M(w) h = -CA^{-1}h + CA^{-1} \int_0^{+\infty} \d F(\T_t w) \d \T_t(w) h \, \d t,
\end{equation}
where $\d \T_t(w)$ denotes the differential of $\T_t$ at $w$, which exists. 
Moreover, if $F$ and $\d F$ are globally Lipschitz continuous, then so is $\d \M$.
\end{theorem}
The proof of \cref{th:M-semilinear} is given in \cref{sec:proof-semilinear}. Its purpose is to guarantee the existence of the control law \cref{eq:feedback-law} for a  wide class of  semilinear systems and provide  easy  conditions under which requirements of \cref{th:global-stab} are met. A notable consequence of \cref{eq:dM} is that the coercivity condition \cref{eq:range-cond} in \cref{th:global-stab} simply reads as
\begin{equation}
 \label{eq:non-resonance}
\operatorname{Range} CA^{-1}B = Z,
\end{equation}
which, in the context of output regulation corresponds to a \emph{non-resonance} condition between $A$ and the zero dynamics of the integrator via the Schur complement (see, e.g., \cite{isidori2003robust, astolfi-praly}).
On the other hand, \cref{th:M-semilinear} also states that $\d \M$ inherits the Lipschitz properties of $F$ and $\d F$.
Now, regarding \cref{hyp:semilinear}, \cref{eq:as-stab-lin} is roughly speaking a sufficient condition under which solutions to the uncontrolled $w$-equation linearized around a given trajectory (also called \emph{first variation equation})
uniformly converge to that trajectory. \Cref{eq:as-stab-lin} is easily verified in (at least) two situations of interest.
\begin{itemize}
    \item The nonlinearity $F$ contributes to the contraction behavior of the $w$-dynamics.
    This is the case for example if $A$ is coercive and $F$ is monotone (i.e. such that $(\d F(w)h, h)_H$ is  nonnegative for all $w, h$ in $H$).
    More generally, the reader may refer to \cite{lohmiller-contraction} in finite dimension or \cite[Chapter V]{temam_infinite-dimensional_2012}
    for the contraction analysis of the system by means of its first variation equation.
    \item Variations of $F$ are small with respect to linear dissipation brought by $A$. More precisely, if $F$ is $K$-Lipschitz continuous, then  $\|\d F(\cdot)\|_{\mathcal{L}(H)}$ is bounded by $K$, hence \cref{eq:as-stab-lin} is satisfied if $(Ah, h)_H \geq \beta \|h\|_H^2$ for some constant $\beta > K$. In the same spirit, under \cref{eq:as-stab-lin} and \cref{eq:non-resonance}, by observing\footnote{
    See \cref{lem:decay-linearized} below for the bound $K/\alpha$.
    } that for each $w$ the integral map in \cref{eq:dM} has operator norm bounded by $K/\alpha$, one deduces from \cref{eq:dM} that \emph{global} uniform coercivity \cref{eq:unif-coercive} holds whenever $K < \alpha$.
\end{itemize}
\begin{rem}[Linear case] If $\A$ is linear, i.e., $F = 0$, and $0$ is in the resolvent set of $A$, then the solution $\M$ to \cref{eq:forwarding-func} is explicitly and uniquely determined as $\M = - CA^{-1}$. This choice corresponds to the linear forwarding approach followed by \cite{marx:hal-02944073,terrand2019adding,marx:hal-03417238}.
\end{rem}

\subsection{The case of unbounded output}  
\label{sec:C-unbounded}
We conclude \cref{sec:main-res} with an informal discussion regarding the  convergence of the output $Cw(t)$ towards the reference when the operator $C$ is unbounded. 
To get some insight on the situation, assume for a moment that the original $w$-system is linear. Then, the solution $\M$ to \cref{eq:forwarding-func} provided by \cref{th:M-semilinear} is a bounded linear operator, and the closed-loop dynamics around the equilibrium are governed by a strongly continuous linear semigroup, which commutes with its generator. Recalling the notation from \cref{sec:semilinear}, it follows that for a \emph{strong} solution $[w, z]$ to \cref{eq:control-system}-\cref{eq:feedback-law}, $w(t) - w^\star$ goes to $0$ in $\D(A)$ endowed the graph norm, hence $Cw(t)$ converges to $\yr$ provided that $C$ is $A$-bounded. In the nonlinear case, the argument breaks down.

Alternatively, one may look for weaker notions of convergence. In many applications, unbounded output operators of interest enjoy an \emph{admissibility} property with respect to the uncontrolled dynamics. In the linear theory (see, e.g., \cite[Section 4.3]{tucsnak2009observation}), $C$ is said to be $A$-admissible if $C$ is $A$-bounded and there exist positive constants $K$ and $T$ such that
\begin{equation}
\label{eq:A-admis}
\int_0^T \|C\mathcal{S}_t w_0\|^2_Z \, \d t \leq K \|w_0\|^2_H, \quad \forall w_0 \in \D(A).
\end{equation}
In the semilinear case, one can deduce from \cref{eq:A-admis} that,
first of all, the output $Cw$ is well-defined in $L^2_{\mathrm{loc}}(0, +\infty; Z)$ even for weak solutions $[w, z]$ to the closed-loop equations \cref{eq:control-system}-\cref{eq:feedback-law}, and secondly, that the output converges ``in average" to the reference:
\begin{equation}
\lim_{\tau \to + \infty} \int_\tau^{T + \tau} \|Cw(t) - \yr\|^2_Z \, \d t = 0
\end{equation}
for any (including weak) solution $[w, z]$ that converges to the equilibrium $[w^\star, z^\star]$ in $H \times Z$. When no semilinear structure is prescribed for the maximal monotone operator $\A$ governing the $w$-dynamics, we  generalize  \cref{eq:A-admis} by assuming  that for any $w_i^{0}$ in  $\D(\A)$ and $f_i$ 
absolutely continuous with derivative in $L^2(0, T; H)$
, the
solution\footnote{Existence and uniqueness of a strong solution is guaranteed by \cite[Chapter 4, Theorem 4.1]{showalter_monotone_2013}. In particular, each $w_i$ is absolutely continuous; hence $\A(w_i)$ is measurable, and so is $Cw_i$ by \cref{eq:forwarding-func}.
}
$w_i$ to $\d w_i / \d t + \A(w) = f_i$ with initial condition $w_i(0) = w_i^0$, $i \in \{1, 2\}$, satisfies
\begin{equation}
\label{eq:incr-admis}
\int_0^T \|Cw_1(t) - Cw_2(t) \|^2_Z \, \d t \leq K \|w_1(0) - w_2(0)\|^2_H + K \int_0^T \|f_1 - f_2\|^2_H \, \d t.
\end{equation}
In that case, 
the same conclusions hold for closed-loop solutions.

\section{Proofs of the main results}\label{sec:proof}
This section is devoted to the proofs of the results of the paper.
In 
what follows, we will investigate well-posedness, global attractivity \eqref{eq:glob-conv} and local exponential stability \eqref{eq:exp-conv} of the closed-loop system \cref{eq:control-system}-\cref{eq:feedback-law}  in the new coordinates $[w, \eta]$ where $\eta$ is given by
\begin{equation}
\label{eq:change-variable}
\eta \triangleq z - \M(w).
\end{equation}
In $[w, \eta]$-coordinates, \cref{eq:control-system}-\cref{eq:feedback-law} may be equivalently rewritten as
\begin{subequations}
\label{eq:cl-new-var}
\begin{align}
\label{eq:cl-new-var-w}
&\dt{w} + \A(w) = BB^*\d\M(w)^*\eta + d, \\
\label{eq:cl-new-var-eta}
&\dt{\eta} + \d\M(w)BB^*\d\M(w)^*\eta = - y_\mathrm{ref} - \d\M(w)d.
\end{align}
\end{subequations}
\Cref{eq:cl-new-var-eta}  is obtained by differentiating \cref{eq:change-variable} and using \cref{eq:forwarding-func} combined with \cref{eq:cl-new-var-w}.
Since $\M$ is continuous, \eqref{eq:change-variable} preserves the topology of the space: given $[w^\star, z^\star]\in H\times Z$, any solution $[w, z]$ of \cref{eq:control-system}-\cref{eq:feedback-law} converges towards some $[w^\star, z^\star]\in H\times Z$ if and only if $[w, \eta]$ converges towards $[w^\star, \eta^\star]$, where $\eta=z-\M(z)$ and $\eta^\star = z^\star-\M(w^\star)$.
Moreover, given a ball of $H$ containing $w^\star$, by local Lipschtz continuity of $\M$, there exist positive constants $K_1$ and $K_2$ such that
\begin{equation}
\label{eq:equiv-z-eta}
K_1 \| [w - w^\star, z - z^\star]\|_{H\times Z} \leq  \| [w - w^\star, \eta - \eta^\star]\|_{H\times Z} \leq K_2  \| [w - w^\star, z - z^\star]\|_{H\times Z}
\end{equation}
whenever $w$ lies in that ball. 
Therefore, well-posedness, global attractivity \eqref{eq:glob-conv}, and local exponential stability \eqref{eq:exp-conv}, of \cref{eq:control-system}-\cref{eq:feedback-law} and \eqref{eq:cl-new-var} are actually equivalent.
Finally, to alleviate notation when needed, we define
\begin{equation}
\K(w) \triangleq \d\M(w) B \in \mathcal{L}(U, Z), \quad \forall w \in H.
\end{equation}

\subsection{Proof of \cref{th:well-posed} (well-posedness)}
\label{sec:proof-wp}
We start by proving that solutions to \cref{eq:cl-new-var} exist at least on a finite time interval (\cref{wpi,wpii} of \cref{th:well-posed}), and then we investigate forward completeness (\cref{wpiii} of \cref{th:well-posed}).

\textbf{Step 1: Local well-posedness.} First, we observe that \cref{eq:cl-new-var} represents a locally Lipschitz perturbation of the following maximal monotone problem:
\begin{equation}
\dt{} [w, \eta] + [\A(w), \eta] = 0.
\end{equation}
Indeed, by letting
\begin{equation}
\label{eq:def-F}
 \mathcal{F}_{d, \yr}[w, \eta] \triangleq \begin{bmatrix} - B\K(w)^*\eta - d \\ - \eta +\K(w)\K(w)^*\eta + y_\mathrm{ref} + \d\M(w)d \end{bmatrix},
 \quad \forall [w, \eta] \in H \times Z,
\end{equation}
we can rewrite \cref{eq:cl-new-var} as follows:
\begin{equation}
\label{eq:cl-cauchy-pb}
\dt{}[w, \eta] + [\A(w), \eta] + \mathcal{F}_{d, \yr}[w, \eta] = 0.
\end{equation}
The nonlinear map $\mathcal{F}_{d, \yr}$ is locally Lipschitz continuous on $H \times Z$. Besides, since $\A$ is maximal monotone on $H$, so is the mapping $[w, \eta] \mapsto [\A(w), \eta]$ on $H \times Z$ (with dense domain $\D(\A) \times Z$). Furthermore, $[\A(0), 0] = 0$. Thus, it follows from \cite[Theorem 7.2]{chueshov_attractor_2002} that for each initial condition $[w_0, \eta_0] \in H \times Z$, there exists a unique maximal weak solution $[w, \eta]$ to \cref{eq:cl-new-var} defined on $[0, T_\mathrm{max})$, with $T_\mathrm{max} \in (0, +\infty]$.
Moreover, if $w_0\in\D(\A)$, $[w, \eta]$ is actually a strong solution.
Recalling that $z = \eta + \M(w)$ and $\M\in \mathcal{C}^1(H, Z)$, we obtain \cref{wpi,wpii}.
Finally, if $T_\mathrm{max}$ is finite, then the norm of $[w(t), \eta(t)]$ must go to $ + \infty$ as $t$ approaches $T_\mathrm{max}$. 

\textbf{Step 2: Sufficient conditions for forward completeness.} Let us prove that if  take $d = 0$ or  $\d \M$ is globally Lipschitz continuous, then $T_\mathrm{\max} = + \infty$ for any initial condition $[w_0, \eta_0] \in H \times Z$. Let $\rho > 0$ to be fixed (large enough) later on. Given a (strong) solution $[w, \eta]$ with initial data $[w_0, \eta_0] \in \D(\A) \times Z$, we have
\begin{multline}
\label{eq:energy-id}
\frac{1}{2}\dt{} \left \{  \|w\|^2_H + \rho \|\eta\|^2_Z \right \} = -(\A(w), w)_H + (B\K(w)^*\eta, w)_H + (d, w)_H  \\ - \rho \|\K(w)^* \eta \|^2_U - \rho(y_\mathrm{ref}, \eta)_Z - \rho (\d\M(w) d, \eta)_Z
\end{multline}
holding a.e. on $(0, T_\mathrm{max})$. Similarly as in our preliminary Lyapunov analysis \cref{eq:pre-lyap}-\cref{eq:lyap}, in order to deal with the term $B\K(w)^*\eta$, we use \cref{eq:nonlin-coer} together with Young's inequality and choose $2\rho \geq \|B\|^2_{\mathcal{L}(U, H)}$ to obtain
\begin{multline}
\label{eq:est-dMd}
\dt{} \left \{  \|w\|^2_H + \rho \|\eta\|^2_Z \right \}  \leq   2\|w\|^2_H  + 2\rho \|\eta \|^2_Z + {\rho} \|y_\mathrm{ref}\|^2_Z +  \| d\|^2_H  + {\rho}\|\d\M(w)d\|^2_Z
\end{multline}
If $d = 0$ or $\d\M$ is globally Lipschitz continuous, the following inequality holds:
\begin{equation}
\label{eq:dMd}
\| \d \M(w)d \|^2_Z \leq K(1 + \|w\|^2_H) \|d\|^2_H
\end{equation}
for some $K > 0$ independent of $w$ or $d$. Combining \cref{eq:est-dMd} and \cref{eq:dMd} yields
\begin{equation}
\label{eq:est-boom-ae}
\dt{} \left \{  \|w\|^2_H + \rho \|\eta\|^2_Z \right \} \leq K' \left \{ \|w\|^2_H + \rho \|\eta \|^2_Z \|\right \} + K' \quad \mbox{a.e.}
\end{equation}
for some $K'> 0$ independent of the initial data. As a strong solution to \cref{eq:cl-new-var},  $[w, \eta]$ is absolutely continuous in $H \times Z$; therefore,  we can deduce from \cref{eq:est-boom-ae} and Gr\"onwall's inequality the following uniform estimate:  for all $t \in [0, T_\mathrm{max})$,
\begin{equation}
\label{eq:no-boom}
\|w(t)\|^2_H + \rho \|\eta(t)\|^2_Z \leq \exp(K't)   \{  \|w_0\|^2_H + \rho \|\eta_0\|^2_Z +1 \}.
\end{equation}
We infer from \cref{eq:no-boom} that the norm of $[w, \eta]$ cannot blow up in finite time; thus, $T_\mathrm{max} = + \infty$. Furthermore, by passing to the limit, we see that \cref{eq:no-boom} is satisfied for weak solutions as well, which means that the same conclusion holds for any initial data in $H \times Z$.
This proves \cref{wpiii} and ends the proof of \cref{th:well-posed}.

\subsection{Proof of \cref{th:global-stab} (output regulation)}
\label{sec:proof-stab} Our main result concerning the output regulation problem is demonstrated here, along with some auxiliary results.

\subsubsection{Outline of the proof and intermediate results}
The proof of \cref{th:global-stab} relies on a series of  lemmas that are given below. Let us first give some insight on the main strategy.
\begin{enumerate}
\item Under the range condition \cref{eq:range-cond}, the operators $\mathcal{K}(w)^* = B^*\d \M(w)^*$ involved in the $\eta$-equation \cref{eq:cl-new-var-eta}  enjoy a  coercivity property that is uniform with respect to $w$, provided that $w$ is small. This is shown in \cref{lem:local-coercivity}.
\item This allows to prove that there exists a (not necessarily invariant) region around the origin of $H \times Z$ where the dynamics generated by \cref{eq:cl-new-var} are strictly contractive. This is stated in \cref{lem:local-contraction}.
\item Another consequence of the local coercivity property is the existence of suitable attracting sets for the closed-loop dynamics \cref{eq:cl-new-var}, as demonstrated in \cref{lem:attraction}.  Furthermore, if the reference $\yr$ and the perturbation $d$ are sufficiently small, then those sets are contained in the contraction region of the previous step. 
\item Finally, existence of a locally exponentially stable equilibrium for \cref{eq:cl-new-var} is shown by using Banach fixed point arguments and properties of $\omega$-limit sets associated with contraction semigroups.
\item Additionally, under the global coercivity condition \cref{eq:unif-coercive}, the attracting sets from \cref{lem:attraction} are in fact globally attractive, hence the equilibrium is also globally asymptotically stable.
\end{enumerate}
Coming back to the $[w, z]$-coordinates and thus to the original closed-loop system \cref{eq:control-system}-\cref{eq:feedback-law} is then straightforward.
In what follows, we shall be careful regarding the dependence of the various neighborhoods involved in the analysis with respect to each parameter. 
Before proceeding further, we introduce some additional notation.  Recalling \cref{eq:def-F} and \cref{eq:cl-cauchy-pb},
 given $[d, \yr]$, we denote by $\Acl$ the nonlinear operator $[\A, \mathrm{id}] + \mathcal{F}_{d, \yr}$ defined on $\D(\A) \times Z$ and associated with  the closed-loop system \cref{eq:cl-new-var} in the $[w, \eta]$-coordinates. Also, for $\rho > 0$, we denote by $\|\cdot \|_{H\times Z, \rho}$ the Hilbert norm on $H \times Z$ given by
\begin{equation}
\label{eq:def-rho-norm}
\|[w, \eta]\|^2_{H \times Z, \rho} \triangleq \|w\|^2_H + \rho \|\eta\|^2_Z, \quad \forall [w, \eta] \in H \times Z.
\end{equation}
All these norms are equivalent. 
We denote by $(\cdot, \cdot)_{H\times Z, \rho}$ the associated scalar product and by $\mathcal{B}_{H\times Z, \rho}([w, \eta], r)$ the open ball of radius $r > 0$ centered at $[w, \eta] \in H \times Z$
with respect to the norm $\|\cdot\|_{H \times Z, \rho}$.
Bearing in mind our Lyapunov analysis in the $[w, z]$-coordinates in \cref{sec:forwarding}, the $\rho$-norm is connected to the Lyapunov function $V$ by $V(w, z) = (1/2)\|[w, z - \M(w)\|^2_{H\times Z, \rho}$.
In the sequel, 
we either take $d = 0$ or assume that $\d\M$ is globally Lipschitz continuous.  Therefore, by virtue of \cref{th:well-posed}, solutions to the closed-loop equations \cref{eq:cl-new-var} are well-defined for all positive time and any initial condition in $H \times Z$.
Recall from \cref{eq:range-cond} in the statement of \cref{th:global-stab} that the following coercivity assumption is in force:
\begin{equation}
\label{eq:coercivity}
\|\K(0)^* z \|^2_U \geq \lambda \|z \|^2_Z, \quad \forall z \in Z.
\end{equation}
The first lemma is a consequence of \cref{eq:coercivity} and continuity of $\d \M$.
\begin{lemma}[Local coercivity]
\label{lem:local-coercivity}
There exist positive constants $\tilde{\lambda}$ and $\tilde{r}$ such that
\begin{equation}
\label{eq:local-coercivity}
\|\K(w)^* z \|^2_U \geq \tilde{\lambda} \|z \|^2_Z, \quad \forall z \in Z, \ \forall w \in \mathcal{B}_H(0, \tilde{r}).
\end{equation}
\end{lemma}
\begin{proof}
 Let $w \in H$ and $z \in Z$. Recall that
$\K(w)^*z = B^*\d\M(0)^*z + B^*[\d\M(w) - \d\M(0)]^*z.
$
Using Cauchy-Schwarz and Young inequalities together with \cref{eq:coercivity}, we get
\begin{equation}
\|\K(w)^*z\|_U^2  \geq \frac{2\lambda}{3} \|z\|^2_H - 2\|B\|^2_{\mathcal{L}(U, H)}  \|[\d\M(w) - \d\M(0)]^* \|^2_{\mathcal{L}(Z, U)}\|z\|^2_Z.
\end{equation}
By continuity of $\d\M$ at $0$, we can choose $\tilde{r} > 0$ such that
\begin{equation}
\| \d \M(w) - \d \M(0) \|_{\mathcal{L}(H, Z)}^2 \leq  \ \frac{\lambda}{6 \|B\|^{2}_{\mathcal{L}(U, H)}}, \quad  \forall w \in \mathcal{B}_H(0, \tilde{r}).
\end{equation}
Thus, by letting $\tilde{\lambda} \triangleq \lambda / 3 > 0$, we obtain the desired inequality \cref{eq:local-coercivity}.
\end{proof}
The following constant appears in the next two lemmas:
\begin{equation}
\label{eq:kappa}
\kappa \triangleq \min \{ \alpha/4, \tilde{\lambda}/4\}.
\end{equation}
Those concern the contraction property of the dynamics governed by \cref{eq:cl-new-var} around the origin and the existence of attractive sets depending on $[d, \yr]$.

\begin{lemma}[Local strong monotonicity] 
\label{lem:local-contraction}
There exists $\rho_0 > 0$ such that the following property holds: for all $\rho \geq \rho_0$, there exists a positive $r_{0, \rho} \leq \tilde{r}$  such that
\begin{multline}
\label{eq:local-cont}
(\Acl[w_1, \eta_1] - \Acl[w_2, \eta_2], [w_1, \eta_1] - [w_2, \eta_2])_{H\times Z, \rho}  \\ \geq  \kappa \|[w_1, \eta_1] - [w_2, \eta_2]\|^2_{H \times Z, \rho}
\end{multline}
for all $[w_1, \eta_1], [w_2, \eta_2] \in \mathcal{B}_{H\times Z}(0, r_{0, \rho}) \cap \D(\A) \times Z$, $d \in \mathcal{B}_H(0, r_{0, \rho})$ and $\yr \in Z$.
\end{lemma}
\begin{proof}
Let $[w_1, \eta_2]$ and $[w_2, \eta_2]$ in $\mathcal{D}(\mathcal{A}) \times Z$. We write $\tilde{w} \triangleq w_1 - w_2$ and $\tilde{\eta} \triangleq \eta_1 - \eta_2$. Then, for any $\rho > 0$,
\begin{multline}
\label{eq:scalar-prod-cl}
(\Acl[w_1, \eta_1] - \Acl[w_2, \eta_2], [\tilde{w}, \tilde{\eta}])_{H\times Z, \rho} 
 \\ = (\A(w_1) - \A(w_2) - B[\K(w_1)^*\eta_1 - \K(w_2)^* \eta_2], \tilde{w})_H \\
+ \rho ( \K(w_1)\K(w_1)^*\eta_1 - \K(w_2)\K(w_2)^*\eta_2 + [\d\M(w_1)  - \d\M(w_2)]d, \tilde{\eta})_Z.
\end{multline}
By adding and removing some terms, \cref{eq:scalar-prod-cl} can be rewritten as follows:
\begin{multline}
\label{eq:Acl-dec}
(\Acl[w_1, \eta_1] - \Acl[w_2, \eta_2], [\tilde{w}, \tilde{\eta}])_{H\times Z, \rho} 
= (\A(w_1) - \A(w_2), \tilde{w})_H \\
+ \rho (\K(w_1)\K(w_1)^*\tilde{\eta}, \tilde{\eta})_Z   + \rho([\K(w_1)\K(w_1)^* 
- \K(w_2)\K(w_2)^*]\eta_2, \tilde{\eta})_Z \\
- (B\K(w_1)^*\tilde{\eta}, \tilde{w})_H 
- (B[\K(w_1) - \K(w_2)]^*\eta_2, \tilde{w})_H  + \rho ([\d\M(w_1) - \d\M(w_2)]d, \tilde{\eta})_Z.
\end{multline}
Assume for a moment that $\|w_i\|_H \leq \tilde{r}/2$,  where $\tilde{r}$ is given by \cref{lem:local-coercivity}.
Therefore, $\|\tilde{w}\|_H \leq \tilde{r}$; and using \cref{eq:nonlin-coer} and \cref{eq:local-coercivity} together with the Lipschtz continuity of $\K$ and $\K(\cdot)\K(\cdot)^*$ on $\mathcal{B}_H(0, \tilde{r})$, we infer from \cref{eq:Acl-dec} that
\begin{multline}
(\Acl[w_1, \eta_1] - \Acl[w_2, \eta_2], [\tilde{w}, \tilde{\eta}])_{H\times Z, \rho} \geq \alpha \|\tilde{w}\|^2_H +  \rho \tilde{\lambda} \|\tilde{\eta}\|^2_Z  \\
- \rho K_1 \|\tilde{w}\|_H \|\eta_2\|_Z \|\tilde{\eta}\|_Z - K_2 \|\tilde{\eta}\|_Z \|\tilde{w}\|_H  
- K_3 \|\eta_2\|_Z \|\tilde{w}\|^2_H -  \rho K_4 \|\tilde{w}\|_H \|d\|_H \|\tilde{\eta}\|_Z.
\end{multline} 
 where the $K_i$ are some positive constants independent of
$[w_i, \eta_i]$, $[d, \yr]$ and $\rho$.
Given $\eps > 0$, we employ Cauchy-Schwarz and Young inequalities to obtain
\begin{multline}
\label{eq:cont-young}
 (\Acl[w_1, \eta_1] - \Acl[w_2, \eta_2], [\tilde{w}, \tilde{\eta}])_{H\times Z, \rho} \geq   \alpha \|\tilde{w}\|^2_H +  \rho \tilde{\lambda} 
\|\tilde{\eta}\|^2_Z \\
- \frac{\rho}{2} \{ K_1 \|\eta_2\|_Z + K_4 \|d\|_H \}\{\|\tilde{w}\|^2_H + \|\tilde{\eta}\|^2_Z\} - \frac{K_2}{2\eps} \|\tilde{\eta}\|^2_Z - \frac{\eps K_2}{2} \|\tilde{w}\|^2_H - K_3 \|\eta_2\|_Z \|\tilde{w}\|^2_H.
\end{multline}
Let $\eps = \alpha/(2 K_2)$ in \cref{eq:cont-young} and define $\rho_0 \triangleq 2K_2/(\tilde{\lambda}\eps)$. For all $\rho \geq \rho_0$,
\begin{multline}
\label{eq:cont-young-bis}
(\Acl[w_1, \eta_1] - \Acl[w_2, \eta_2], [\tilde{w}, \tilde{\eta}])_{H\times Z, \rho} \geq \frac{\alpha}{2} \|\tilde{w}\|^2_H + \frac{\rho \tilde{\lambda}}{2} \|\tilde{\eta}\|^2_Z \\ - \frac{\rho}{2} \{ K_1 \|\eta_2\|_Z + K_4 \|d\|_H \}\{\|\tilde{w}\|^2_H + \|\tilde{\eta}\|^2_Z\} - K_3 \|\eta_2\|_Z \|\tilde{w}\|^2_H.
\end{multline}
Let $\rho \geq \rho_0$. We infer from \cref{eq:cont-young-bis} that there exists $r_{0, \rho} \leq \tilde{r}/2$ such that
\begin{equation}
\label{eq:local-contr}
(\Acl[w_1, \eta_1] - \Acl[w_2, \eta_2], [\tilde{w}, \tilde{\eta}])_{H\times Z, \rho} \geq   \frac{\alpha}{4} \|\tilde{w}\|^2_H +  \frac{\rho \tilde{\lambda}}{4} \|\tilde{\eta}\|^2_Z
\end{equation}
as long as $[w_i, \eta_i] \in \mathcal{B}_{H\times Z}(0, r_{0, \rho})$ and $d \in \mathcal{B}_H(0, r_{0, \rho})$, which completes the proof. 
\end{proof}

\begin{lemma}[Absorbing balls]
\label{lem:attraction}
 There exists  $\rho_1 > 0$ such that the following property holds:
for any $\rho \geq \rho_1$, there exist positive   $r_{1,\rho}$ and $K_\rho$ such that, if $[d, \yr] \in \mathcal{B}_{H \times Z}(0, r_{1, \rho})$, then the estimate
\begin{multline}
\label{eq:attraction}
\|[w(t), \eta(t)]\|^2_{H \times Z, \rho}  \\ \leq  \exp(- \kappa t) \|[w_0, \eta_0]\|^2_{H \times Z, \rho}  + K_\rho \{1 - \exp(-\kappa t)\} \|[d, \yr]\|^2_{H\times Z}
\end{multline}
holds for any solution $[w, \eta]$ to \cref{eq:cl-new-var} with initial data $[w_0, \eta_0]$ in $\mathcal{B}_{H \times Z}(0, r_{1, \rho})$.
Furthermore, if \cref{eq:local-coercivity} holds globally,  then \cref{eq:attraction} is true for all initial data in $H \times Z$.
\end{lemma}

\begin{proof} 
All formal computations performed below are justified by considering appropriate sequence of strong solutions as provided by \cref{th:well-posed} and then passing to the limit at the very end. First, let 
\begin{equation}\rho_1 \triangleq  \|B\|_{\mathcal{L}(U, H)}^2 \max \{1, 2\alpha^{-1}\}.
\end{equation}
Let $\rho \geq \rho_1$ be fixed and let $[w, \eta]$ be a solution to \cref{eq:cl-new-var} with initial condition $[w_0, \eta_0]$.
As in the proof of \cref{th:well-posed}, using \cref{eq:energy-id}  we obtain
\begin{equation}
\label{eq:ineq-diff-d-y}
\dt{}   \|[w, \eta]\|^2_{H\times Z, \rho} \leq k_\rho  \{  \|[w, \eta]\|^2_{H\times Z, \rho} + \|y_\mathrm{ref}\|^2_Z +  \|d\|^2_H  \}
\end{equation}
for some $k_\rho > 0$ independent of $[w_0, \eta_0]$ and $[d, y_\mathrm{ref}]$. It follows from \cref{eq:ineq-diff-d-y} that
\begin{equation}
\label{eq:est-01}
\|[w(t), \eta(t)]\|^2_{H\times Z, \rho} \leq  \exp(k_\rho)  \|[w_0, \eta_0]\|^2_{H\times Z, \rho} +  \exp(k_\rho)\{ \|\yr\|^2_Z + \|d\|^2_H\}
\end{equation}
for all $t \in [0, 1]$.
As a consequence of \cref{eq:est-01}, there exists a positive constant $r_{\rho}$ such that $\max \{ \|w_0\|_H, \|\eta_0\|_Z, \|d\|_H, \|y_\mathrm{ref}\|_Z \} \leq r_\rho$ implies
$\|w(t)\|_H \leq \tilde{r}$ for all $t \in [0, 1]$,
where $\tilde{r}$ is defined in \cref{lem:local-coercivity}. Therefore, for such initial data,  we can use \cref{eq:local-coercivity} in \cref{eq:energy-id} to refine our previous estimate: for all positive $\eps$ and $\mu$, on  $(0, 1)$ we have
\begin{multline}
\label{eq:energy-id-young}
\frac{1}{2}\dt{}  \|[w, \eta]\|^2_{H\times Z, \rho} \leq - \alpha \|w\|^2_H + \frac{1}{2\eps} \|B\|_{\mathcal{L}(U, H)}^2 \|\K(w)^*\eta \|^2_U + \eps \|w\|^2_H  \\ + \frac{1}{2\eps} \|d\|^2_H
 - \rho  \|\K(w)^* \eta \|^2_U  + \frac{\rho}{2 \mu} \|\yr\|^2_Z +  \frac{\rho}{2\mu} \|\d \M(w) d \|^2_Z + \rho \mu \|\eta\|^2_Z.
\end{multline}
First, recall that we have either  $\d \M$ globally Lipschitz continuous or $d = 0$. Hence, there exists $k > 0$ independent of $[w_0, \eta_0]$ and $[d, \yr]$ such that $\|\d\M(w)d\|^2_Z \leq k \|d\|^2_H$.
By choosing $\eps = \alpha / 2$ and $ \mu = \tilde{\lambda}/ 4 $, we deduce from \cref{eq:energy-id-young} the following differential inequality, valid on $(0, 1)$:
\begin{equation}
\label{eq:est-good-signs}
\frac{1}{2}\dt{}  \|[w, \eta]\|^2_{H\times Z, \rho}  \leq - \frac{\alpha}{2} \|w\|^2_H - \frac{\rho \tilde{\lambda}}{4} \|\eta\|^2_Z  + \frac{\rho}{2 \mu} \|\yr\|^2_Z + \left \{ \frac{\rho k}{2 \mu} + \frac{1}{2 \epsilon} \right \} \|d \|^2_H
\end{equation}
Applying Gr\"onwall's inequality to \cref{eq:est-good-signs} yields
\begin{multline}
\label{eq:estimate-finite-time}
\|[w(t), \eta(t)]\|^2_{H\times Z, \rho} \\ \leq  \exp(- \kappa t)   \|[w_0, \eta_0]\|^2_{H\times Z, \rho} +  K_\rho(1 - \exp(- \kappa t))\| [d, \yr]\|^2_{H\times Z}
\end{multline}
for all $t \in [0, 1]$, where  $\kappa$ is defined in \cref{eq:kappa} and $K_\rho \geq 1$ is some constant independent of $[w_0, \eta_0]$ and $[d, \yr]$.
Next, by norm equivalence, there exists $r_{1, \rho} > 0$ such that the following implication holds: if $[w_0, \eta_0]$ and  $[d, \yr]$ are in $\mathcal{B}_{H \times Z}(0, r_{1, \rho})$, then
\begin{equation}
\|[w_0, \eta_0] \|_{H \times Z, \rho} \leq 2^{-1/2} r_\rho, \quad \mbox{and} \quad \|[d, \yr]\|_{H \times Z} \leq (2 K_\rho)^{-1/2} r_\rho.
\end{equation}
Now, we claim that
the estimate \cref{eq:estimate-finite-time} remains valid for all $t \geq 0$. Indeed,
\cref{eq:estimate-finite-time} shows that $\|[w(1), \eta(1)]\|_{H\times Z, \rho} \leq r_\rho$. Therefore, by definition of $r_\rho$, we infer from the estimate \cref{eq:est-01} applied to the initial data $[w(1), \eta(1)]$ that $\|w(t)\|_H \leq \tilde{r}$ for all $t \in [1, 2]$.
As a consequence, the differential inequality \cref{eq:est-good-signs} is valid on $(0, 2)$; hence, \cref{eq:estimate-finite-time} holds on $[0, 2]$, with in particular $\|[w(2), \eta(2)\|_{H \times Z, \rho} \leq r_\rho$, and so on. The conclusion readily follows by induction.

Moreover, if it is assumed that \cref{eq:local-coercivity} holds for all $w \in H$, then \cref{eq:est-good-signs} is valid on $(0, + \infty)$ whatever the initial condition, so that \cref{eq:est-01} immediately holds for all $t \geq 0$. In this case, no additional condition on $[d, \yr]$ is required.
\end{proof}

\subsubsection{Proof of \cref{th:global-stab}} We can now prove the main result.

\textbf{Step 1: Setting all neighborhoods.} Pick $\rho$ such that $\rho \geq \max \{ \rho_0, \rho_1 \}$ as in \cref{lem:local-contraction,lem:attraction}. In the sequel, $[d, \yr]$ is assumed to lie in the intersection of $\mathcal{B}_H(0, r_{0, \rho}) \times Z$ and 
$\mathcal{B}_{H \times Z}(0, r_{1, \rho})$, so that the lemmas apply.
Now that $\rho$ is fixed, we will omit the dependence on $\rho$ in further notation.   \cref{lem:local-contraction} provides a neighborhood $\mathfrak{K}$ of the origin in $H \times Z$, which we will refer to as the contraction region,  where \cref{eq:local-cont} holds. On the other hand, according to \cref{eq:attraction} in \cref{lem:attraction}, the set
\begin{equation} 
\label{eq:def-Vdyr}
\mathfrak{V}_{d, \yr} \triangleq \mathcal{B}_{H \times Z, \rho}(0, K_\rho^{1/2}\|[d, \yr]\|_{H \times Z})
\end{equation}
attracts all solutions to \cref{eq:cl-new-var} originating from $\mathcal{B}_{H \times Z}(0, r_{1, \rho})$. By norm equivalence, any sufficiently small ball for the $\rho$-norm that is centered at the origin is contained in both $\mathfrak{K}$ and $\mathcal{B}_{H \times Z}(0, r_{1, \rho})$. That being said, as a consequence of \cref{eq:attraction} and \cref{eq:def-Vdyr}, there exists positive numbers $\delta$ and $r$ such that, having let $\mathcal{B} \triangleq \mathcal{B}_{H \times Z, \rho}(0, \delta)$, the following properties hold for any $[d, \yr] \in \mathcal{B}_{H \times Z}(0, r)$:
\begin{itemize}
\item The closure of the corresponding attracting set $\mathfrak{V}_{d, \yr}$ is contained in $\mathcal{B}$;
\item
Solutions  to \cref{eq:cl-new-var} with initial data in $\mathcal{B} $ remain in $\mathfrak{K}$.
\end{itemize}
If in addition we assume that \cref{eq:unif-coercive} holds, i.e., \cref{eq:local-coercivity} holds globally, then by  \cref{lem:attraction},
\begin{itemize}
\item The set $\mathfrak{V}_{d, \yr}$ attracts all solutions to \cref{eq:cl-new-var}, whatever the initial data.
\end{itemize}
In what follows, we omit the dependence on $[d, \yr]$ in the notation and we denote by $\{ \tilde{\T}_t \}$ the evolution semigroup associated with \cref{eq:cl-new-var}.
Then, since solutions originating from (the non-empty open set) $\mathcal{B}$ remain in the contraction region $\mathfrak{K}$, we infer from \cref{eq:local-cont} together with a density argument that
\begin{equation}
\label{eq:cont-prop}
\| \tilde{\T}_t[w_1, \eta_1] - \tilde{\T}_t[w_2, \eta_2]\|_{H \times Z, \rho} \leq \exp(- \kappa t) \|[w_1, \eta_1] - [w_0, \eta_0]\|_{H\times Z, \rho}
\end{equation}
for all $t \geq 0$ and $[w_i, \eta_i] \in \mathcal{B}$, $i \in \{ 1, 2\}$.

\textbf{Step 2: Existence of a fixed point.} Pick an arbitrary $[w_0, \eta_0] \in \mathcal{B}$. By \cref{eq:cont-prop} and a usual contraction argument, we see that
\begin{equation}
\{ \tilde{\T}_{n}[w_0, \eta_0] \}_{n \geq 0} ~\mbox{is a Cauchy sequence in}~ H \times Z
\end{equation}
and  converges  to a fixed point $[w_0^\star, \eta_0^\star]$ of the nonlinear operator $\tilde{\T}_1$. Now, consider the $\omega$-limit set $\omega([w_0, \eta_0])$ of $[w_0, \eta_0]$ with respect to the evolution semigroup $\{\tilde{\T}_t\}$. By the sequential characterization of $\omega$-limit sets (see \cite[Lemma 2.1, p.19]{chueshov_introduction_2002}), we observe that $[w_0^\star, \eta_0^\star] \in \omega([w_0, \eta_0])$, which means that $\omega([w_0, \eta_0])$ is non-empty.
Moreover,  it is positively invariant by definition. Finally, since $\mathfrak{V}_{d, \yr}$ attracts all solutions originating from $\mathcal{B}$, we must have
\begin{equation}
\omega([w_0, \eta_0]) \subset \overline{\mathfrak{V}_{d, \yr}} \subset \mathcal{B}.
\end{equation}
By following \textit{verbatim}\footnote{
The only difference with \cite[Theorem 1]{dafermos_asymptotic_1973}
is that $\{\tilde{\T}_t\}$ is a contraction only on a region (containing the $\omega$-limit set) that is not \textit{a priori} positively invariant. However, in order to obtain the isometry property, contraction is only needed on the points of the $\omega$-limit set. 
} the proof of \cite[Theorem 1]{dafermos_asymptotic_1973}, we obtain that for each $t \geq 0$, $\tilde{\T}_t$ is an isometry on $\omega([w_0, \eta_0])$. On the other hand,  for positive $t$, $\tilde{\T}_t$ is a strict contraction on $\omega([w_0, \eta_0])$; thus, $\omega([w_0, \eta_0])$ must be reduced to the singleton $\{[w_0^\star, \eta_0^\star]\}$. By invariance of the $\omega$-limit set, $[w^\star_0, \eta^\star_0]$ is fixed by the semigroup $\{\tilde{\T_t}\}$. Moreover, it follows from \cref{eq:cont-prop} that  $[w_0^\star, \eta_0^\star]$ is the unique fixed point of $\{\tilde{\T_t}\}$ in $\mathcal{B}$ and is exponentially attractive in $\mathcal{B}$. Thus, we write $[w^\star, \eta^\star] \triangleq [w^\star_0, \eta^\star_0]$.

\textbf{Step 3: The fixed point lies in the domain.} 
We now prove that $w^\star \in \D(\A)$. Let $\eps > 0$ sufficiently small and consider the ball 
$\mathcal{C} \triangleq {\mathcal{B}_{H\times Z, \rho}([w^\star, \eta^\star], \eps)} \subset \mathfrak{K}.
$
Since $\D(\A)$ is dense in $H$ and $\mathcal{C}$ has non-empty interior, we can pick some $[w_0, \eta_0] \in [\D(\A) \times Z] \cap \mathcal{C}$. Let $[w(t), \eta(t)] \triangleq \tilde{\T}_t[w_0, \eta_0]$. As a strong solution to \cref{eq:cl-new-var}, $[w, \eta]$
is differentiable  in $H \times Z$ for a.e. time and
\begin{equation}
\dt{}[w, \eta] + \Acl[w, \eta] = 0 \quad \mbox{a.e.}
\end{equation}
Besides, we infer from \cref{eq:cont-prop} that $\mathcal{C}$ is positively invariant so that $\{\tilde{\T}_t\}$ restricted to $\mathcal{C}$ is still a well-defined contraction semigroup. Thus, we can apply \cite[Theorem 1.4]{crandall_semi-groups_1969} to obtain
\begin{equation}
\|\Acl[w, \eta]\|_{H \times Z, \rho} \leq \|\Acl[w(t_0), \eta(t_0)]\|_{H \times Z, \rho} \quad \mbox{a.e. on}~ (t_0, +\infty)
\end{equation}
for some $t_0 \geq 0$. In particular, 
$\|\A(w)\|_H$ is bounded  a.e.  on $(t_0, +\infty)$.
On the other hand, $w(t)$ converges to $w^\star$ in $H$ when $t$ goes to $+ \infty$. Therefore,  it follows from $\A$ being maximal monotone and \cite[Lemma 2.3]{crandall_semi-groups_1969} that $w^\star \in \D(\A)$.

\textbf{Step 4: Conclusion.} Since $[w^\star, \eta^\star]$ belongs to $\D(\A) \times Z$ and is fixed by $\{\tilde{\T}_t \}$, $\Acl[w^\star, \eta^\star] = 0$. We come back to the original $[w, z]$ coordinates by letting $z^\star \triangleq \eta^\star + \M(w^\star)$ and
$\mathcal{N} \triangleq \{ [w, \eta + \M(w)], [w, \eta] \in \mathcal{B} \}.
$
Then, $[w^\star, z^\star]$ belongs to $[\D (\A) \times Z] \cap \mathcal{N}$ and is an equilibrium for \cref{eq:control-system}-\cref{eq:feedback-law}; hence, $Cw^\star = \yr$. Because $\M$ vanishes at $0$ and is continuous, $\mathcal{N}$ is indeed a neighborhood of $0$. Local exponential stability of $[w^\star, z^\star]$ with decay rate $\kappa$ and bassin of attraction containing $\mathcal{N}$ follows from \cref{eq:cont-prop} and our prior remarks regarding the change of coordinates; the constant $M$ in \cref{eq:exp-cont} comes from \cref{eq:equiv-z-eta} and equivalence with the $\rho$-norm. Additionally, under the stronger condition \cref{eq:unif-coercive}, $[w^\star, z^\star]$ is globally asymptotically stable. The proof is now complete.

\subsection{Proof of \cref{th:M-semilinear}
(semilinear case)}
\label{sec:proof-semilinear}
We first give some auxiliary results valid under the hypotheses of \cref{sec:semilinear}, and then we prove \cref{th:M-semilinear}. 

\subsubsection{Preliminaries}  We endow $\D(A)$ with the graph norm. First, we claim that \emph{weak} solutions (in the sense of nonlinear semigroup theory) to 
\begin{equation}
\label{eq:original-w}
\dt{w} + \A(w) = 0.
\end{equation}
coincide with \emph{mild} solutions (in the sense of perturbation of linear equation) -- see, e.g., \cite[Theorem 11.1.5]{CurZwa20book}. In other words, solutions $t \mapsto \T_t w_0$ to \cref{eq:original-w} with initial data $w_0 \in H$ are characterized by
\begin{equation}
\label{eq:variation-const}
\T_t w_0 = \S_t w_0 - \int_0^t  \S_{t - s}F(\T_s w_0) \, \d s, \quad \forall t\geq 0.
\end{equation}
\emph{Strong} solutions $w$ to \cref{eq:original-w}
enjoy the regularity $w \in \mathcal{C}^1(\R_+, H) \cap \mathcal{C}(\R_+, \D(A))$.
\cref{lem:int-formula} given below allows us to circumvent the possible unboundedness of the output operator $C$ and,  together with the $A$-boundeness of $C$, guarantees that $\M$ given by the limit in \cref{eq:M-int}
is a well-defined nonlinear mapping on the whole space $H$.

\begin{lemma}[Integral formula]
\label{lem:int-formula}
For any $w_0 \in H$,
\begin{equation}
\label{eq:int-formula}
 \lim_{\tau \to + \infty} \int_0^\tau \mathcal{T}_t w_0 \, \mathrm{d}t = A^{-1}w_0 - A^{-1}\int_0^{+\infty} F(\mathcal{T}_tw_0) \, \mathrm{d}t.
\end{equation}
In particular, $ \lim_{\tau \to + \infty} \int_0^{\tau} \mathcal{T}_t(w_0) \, \mathrm{d}t$ belongs to $\mathcal{D}(A)$.
\end{lemma}

\begin{proof}
Let $\tau > 0$ and $w_0 \in \D(A)$. Then,  $ t \mapsto \T_t w_0 \in \mathcal{C}^1(\R_+, H) \cap \mathcal{C}(\R_+, \D(A))$ solves \cref{eq:original-w} in a classical sense. After applying the bounded operator $A^{-1}$ to the differential equation, one obtains
\begin{equation}
\label{eq:diff-A-1}
\dt{} [A^{-1}\T_t w_0]  +\T_t w_0 + A^{-1} F(\T_t w_0) = 0, \quad \forall t \geq 0.
\end{equation}
Integrating \cref{eq:diff-A-1} over $(0, \tau)$ yields
\begin{equation}
\label{eq:int-A-1}
\int_0^\tau \T_t w_0 \, \d t = A^{-1} w_0 - A^{-1} \T_\tau w_0 - A^{-1} \int_0^\tau F(\T_t w_0) \, \d t.
\end{equation}
Because the map $w_0 \mapsto (t\mapsto \T_t w_0)$ is continuous from $H$ to $\mathcal{C}([0, \tau], H)$, a density argument shows that \cref{eq:int-A-1} is actually valid for all $w_0 \in H$. To obtain \cref{eq:int-formula}, it then suffices to let $\tau \to + \infty$, having in mind that $\T_\tau w_0 \to 0$ in $H$ and the integral at the right-hand side of \cref{eq:int-A-1} is absolutely convergent.
\end{proof}

Now, let us come back to the uncontrolled $w$-equation \cref{eq:original-w}, which we linearize around a given trajectory $\{\T_t w_0, t \geq 0 \}$:
\begin{equation}
\label{eq:first-var}
\frac{\mathrm{d}v}{\mathrm{d}t} + Av + \mathrm{d}F(\mathcal{T}_t w_0) v = 0.
\end{equation}
Equation \cref{eq:first-var} is linear but non-autonomous in general. Given $h \in H$,  \cref{eq:first-var} possesses a unique mild solution $t \mapsto v(t) = v( t ; w_0, h)$ satisfying $v(0) = h$ and $v \in \mathcal{C}(\R_+, H)$ (see \cite[Theorem 1.2, p.184]{pazy_semigroups_2012}).  Given $w_0$ in $H$, the next lemma states that the nonlinear operators $\T_t$ are all differentiable at $w_0$ and their differentials coincide with the evolution family associated with \cref{eq:first-var}. 
This is a classical fact for sufficiently smooth nonlinear dynamics; here, we give a proof that matches our particular set of hypotheses.

\begin{lemma}[Differentiability of the semigroup]
\label{lem:diff-sg}
Each operator $\mathcal{T}_t$ is Fr\'echet differentiable. Furthermore, for any $t\geq 0$ and $w_0 \in H$, the differential $\mathrm{d}\mathcal{T}_t(w_0)$ is given by
\begin{equation}
\mathrm{d}\mathcal{T}_t(w_0)h = v(t; w_0, h) \quad \mbox{for all}~ h \in H,
\end{equation}
where $t \mapsto v(t; w_0, h)$ is the unique mild solution to \cref{eq:first-var} with initial data $h$.
\end{lemma}

\begin{proof}
Let $\tau \geq 0$ and $w_0 \in H$. It is clear that the mapping $h \mapsto v(\tau; w_0, h)$  is linear; it is also continuous by  \cite[Theorem 1.2, p.184]{pazy_semigroups_2012}.
First, since $F$ is differentiable, the following Taylor formula holds:
\begin{equation}
\label{eq:taylor}
F(a + b) - F(a) = \mathrm{d}F(a) b + R(a, b), \quad \forall a, b \in H,
\end{equation}
with $ R(a, b) = o(\|b\|_H)$ when $\|b\|_H \to 0$ for fixed $a \in H$. Now, take a nonzero $h \in H$. 
Combining the variation of the constant formula \cref{eq:variation-const} with \cref{eq:taylor} leads to
\begin{multline}
\label{eq:diff-reste}
\mathcal{T}_t(w_0 + h) - \mathcal{T}_t w_0 = \mathcal{S}_t h - \int_0^t \mathcal{S}_{t - s} \mathrm{d}F(\mathcal{T}_s w_0)\{\mathcal{T}_s(w_0 + h) - \mathcal{T}_s w_0  \} \, \mathrm{d}s \\  - \int_0^t \mathcal{S}_{t - s} R(\mathcal{T}_s w_0, \mathcal{T}_s(w_0 + h) - \mathcal{T}_s w_0) \, \mathrm{d}s.
\end{multline}
Omitting the dependence on $w_0$ and $h$ for the moment, we write $v(t) \triangleq v(t; w_0, h)$ and $\mathcal{R}(t) \triangleq  \mathcal{T}_t(w_0 + h) - \mathcal{T}_tw_0 - v(t)$. 
Now, taking the difference between \cref{eq:diff-reste} and the integral identity satisfied by $v$ as a mild solution to \cref{eq:first-var}, one obtains
\begin{equation}
\label{eq:diff-reste-bis}
\mathcal{R}(t) = \int_0^t \mathcal{S}_{t - s}\mathrm{d}F(\mathcal{T}_s w_0) \mathcal{R}(s) \, \mathrm{d}s - \int_0^t \mathcal{S}_{t - s}  R(\mathcal{T}_s w_0, \mathcal{T}_s(w_0 + h) - \mathcal{T}_s w_0) \, \mathrm{d}s
\end{equation}
holding for all $0 \leq t \leq \tau$. Next, because $\d F$ is continuous on the compact set $\{\mathcal{T}_sw_0, 0 \leq s \leq \tau \}$, $\|\d F(\T_s w_0) \|_{\mathcal{L}(H)}$ is bounded by some $m$ independent of $h$.
Besides, $\|\S_s \|_{\mathcal{L}(H)} \leq 1$ for all $s \geq 0$. Therefore, it follows from \cref{eq:diff-reste-bis} that for all $0 \leq t \leq \tau$, 
\begin{equation}
\label{eq:estimate-rest}
\|\mathcal{R}(t)\|_H \leq m \int_0^t \|\mathcal{R}(s) \|_H \, \mathrm{d}s + \int_0^t \| R(\mathcal{T}_s w_0, \mathcal{T}_s(w_0 + h) - \mathcal{T}_s w_0) \|_H \, \mathrm{d}s.
\end{equation}
Using Gr\"onwall's inequality in its integral form, we deduce from \cref{eq:estimate-rest} that
\begin{equation}
\label{eq:est-final-R}
\frac{\|\mathcal{R}(\tau)\|_H}{\|h\|_H} \leq \exp(m\tau) \int_0^\tau \frac{ \| R(\mathcal{T}_s w_0, \mathcal{T}_s(w_0 + h) - \mathcal{T}_s w_0) \|_H}{\|h\|_H} \, \mathrm{d}s.
\end{equation}
To obtain the desired differentiability property, it suffices to show that the right-hand side of \cref{eq:est-final-R} converges to $0$ as $\|h\|_H$ goes to $0$. This is done using Lebegue's dominated convergence theorem. Let $s \in [0, \tau]$. 
By the contraction property of $\{\T_t\}$, for any nonzero $h \in H$ such that $\mathcal{T}_s(w_0 + h) - \mathcal{T}_sw_0$ is nonzero,  we have
\begin{equation}
\label{eq:R-frac}
\frac{ \| R(\mathcal{T}_s w_0, \mathcal{T}_s(w_0 + h) - \mathcal{T}_s w_0) \|_H}{\|h\|_H} \leq \frac{\| R(\mathcal{T}_s w_0, \mathcal{T}_s(w_0 + h) - \mathcal{T}_s w_0) \|_H}{ \|\mathcal{T}_s(w_0 + h) - \mathcal{T}_sw_0\|_H},
\end{equation}
and if  $\mathcal{T}_s(w_0 + h) - \mathcal{T}_sw_0 = 0$, then $R(\mathcal{T}_s w_0, \mathcal{T}_s(w_0 + h) - \mathcal{T}_s w_0) = 0$. Either way, when $\|h\|_H \to 0$, $\| \T_s (w_0 + h) - \T_ sw_0\|_H \to 0$, and by definition of the residual term $R$, the right-hand side of \cref{eq:R-frac} must converge to $0$ as well. 
To conclude the proof, let us estimate the left-hand side of \cref{eq:R-frac} uniformly with respect to $h$. 
We observe that the set of all points $\T_s(w_0 + h)$, $0 \leq s \leq \tau$, $\|h\|_H \leq 1$, is contained in some open ball, on which $F$ is $K$-Lipschitz continuous for some $K > 0$. Thus, using \cref{eq:taylor}, one obtains that the left-hand side of \cref{eq:R-frac} is smaller than $2K$.
\end{proof}
We continue by establishing exponential decay of solutions to \cref{eq:first-var}.
\begin{lemma}[Stability of the linearized equation]
\label{lem:decay-linearized}
 Let $w_0 \in H$. For any $h \in H$, the solution $t \mapsto v(t; w_0, h)$ to \cref{eq:first-var} with initial data $h$ satisfies
\begin{equation}
\label{eq:decay-linearized}
\|v(t; w_0, h) \|_H \leq \exp(- \alpha t) \|h\|_H, \quad \forall t \geq 0.
\end{equation}
\end{lemma}
\begin{proof}
Let $w_0$ and $h$ in $H$. Let $\tau \geq 0$ and $f \triangleq t \mapsto \d F(\T_t w_0)v(t) \in \mathcal{C}([0, \tau], H)
$. Then, pick sequences $\{h_n\} \subset \D(A)$ and absolutely continuous $\{ f_n\}$
such that
$h_n \to h$ in $H$ and $f_n \to f$ in $L^2(0, \tau ; H)
$. For each $n$,  there exists a unique strong solution $v_n$ to
$ 
\d {v_n}/ \d t + Av_n + f_n(t) = 0
$
satisfying the initial condition $v_n(0) = h_n$. Furthermore,
\begin{equation}
\label{eq:id-ener-vn}
\frac{1}{2} \dt{} \|v_n\|^2_H = -(Av_n, v_n)_H - (f_n, h)_H \quad \mbox{ a.e. on}~ (0, \tau),
\end{equation}
and $v_n$ converges to $v$ in $\mathcal{C}([0, \tau], H)$. Plugging \cref{eq:as-stab-lin}  into \cref{eq:id-ener-vn} yields
\begin{equation}
\label{eq:id-ener-dF}
\frac{1}{2} \dt{} \|v_n\|^2_H  \leq - \alpha \|v_n\|^2_H + (\d F(\T_t w_0)v_n - f_n, v_n)_H \quad \mbox{ a.e. on}~ (0, \tau).
\end{equation}
We deduce from \cref{eq:id-ener-dF} that for all $ 0 \leq t \leq \tau$,
\begin{equation}
\label{eq:id-artif}
\|v_n(t)\|^2_H \leq \exp(-2\alpha t)\|h_n\|^2_H + \frac{1}{2\alpha} \int_0^\tau | (\d F(\T_s w_0)v_n(s) - f_n(s), v_n(s))_H| \, \d s.
\end{equation}
As $n$ goes to $+ \infty$, the integral term in \cref{eq:id-artif} tends to $0$ and we obtain the desired result by passing to the limit.
\end{proof}

\subsubsection{Proof of \cref{th:M-semilinear}} 
Now that we have established that all objects in the statement of the theorem are well-defined, we can give the proof of the result. 
Recall that $CA^{-1}$ is a bounded linear operator. 
In view of the formula \cref{eq:int-formula}, 
the desired properties of $\M$ and $\d \M$ readily follow from those of the mappings
$
w  \mapsto \int_0^{+\infty} F(\T_t w) \, \d t
$
and
$
w  \mapsto \int_0^{+\infty} \d F(\T_t w) \d \T_t (w) \, \d t,
$
which we investigate next.

\textbf{Step 1: Differentiability.} 
Let $w_0 \in H$. It suffices to prove that
\begin{equation}
\label{eq:quot-diff-M}
 \int_0^{+\infty} \frac{\|F(\mathcal{T}_t(w_0 + h)) - F(\mathcal{T}_t w_0) - \mathrm{d}F(\mathcal{T}_tw_0) \mathrm{d}\mathcal{T}_t(w_0)h \|_H}{\|h\|_H} \, \mathrm{d}t \to 0
\end{equation}
when $\|h\|_H$ goes to $0$. We use Lebesgue's dominated convergence theorem. Since $F$ and $\mathcal{T}_t$ are differentiable, by the chain rule, the integrand in \cref{eq:quot-diff-M} converges to $0$ pointwise.
Let us now find some integrable dominating function.
As in the proof of \cref{lem:diff-sg}, we can find some open ball where $F$ is $K$-Lipschitz and which contains the set of all $\T_t(w_0 + h)$, $\|h\|_H \leq 1$, $t \geq 0$. Thus, it follows from \cref{eq:exp-cont} and \cref{eq:decay-linearized} that
the integrand in \cref{eq:quot-diff-M} is dominated by $t \mapsto 2K\exp(- \alpha t)$, which is integrable.

\textbf{Step 2: Lipschitz continuity of the differential.}
Pick two elements $w_1$ and $w_2$ in $H$. We write $R \triangleq \max \{ \|w_1\|_H, \|w_2\|_H \}$. Subsequent estimates are motivated by the following decomposition:
\begin{multline}
\label{eq:decomposition}
\d F(\T_t w_1) \d \T_t (w_1) - \d F(\T_t w_2) \d \T_t (w_2)  \\ = \d F(\T_t w_1) [ \d \T_t(w_1) - \d \T_t(w_2)]  + [\d F (\T_t w_1) - \d F(\T_t w_2)] \d \T_t (w_2).
\end{multline}
First, $\T_t w_1$ and $\T_t w_2$ must remain in $\mathcal{B}_H(0, R)$, where $\d F$ is, say, $K_R$-Lipschitz continuous.  
 Since $\d F(0) = 0$, for all $t \geq 0$ we have
\begin{equation}
\label{eq:est-Ttw1}
\|\d F(\T_t w_1) \|_{\mathcal{L}(H)} \leq K_R \| \T_t w_1 \|_H \leq K_R \|w_1\|_H.
\end{equation}
Now, let us estimate $ \d \T_t(w_1) - \d \T_t(w_2)$ in operator norm. Pick $h \in H$. In what follows, we denote by $v(t)$ the difference
$
v(t) \triangleq \mathrm{d}\mathcal{T}_t(w_1)h - \mathrm{d}\mathcal{T}_t(w_2)h
$. Then, $v(0) = 0$ and $v$ is a (mild) solution to the following non-automonous equation:
\begin{equation}
\frac{\mathrm{d}v}{\mathrm{d}t} + Av + [\mathrm{d}F(\mathcal{T}_tw_1) \mathrm{d}\mathcal{T}_t (w_1) - \mathrm{d}F(\mathcal{T}_t w_2)\mathrm{d}\mathcal{T}_t (w_2)]h = 0,
\end{equation}
which can be rewritten as
\begin{equation}
\label{eq:new-equation-v}
\frac{\mathrm{d}v}{\mathrm{d}t} + Av + \mathrm{d}F(\mathcal{T}_tw_1)v + [\mathrm{d}F(\mathcal{T}_t w_1) - \mathrm{d}F( \mathcal{T}_t w_2)] \d \T_t(w_2) h = 0.
\end{equation}
Justifications for the formal computations performed below are similar to those in the proof of \cref{lem:decay-linearized}; they are omitted here. Taking the scalar product in $H$ of \eqref{eq:new-equation-v} with $v$ and using \cref{eq:as-stab-lin} along with Cauchy-Schwarz and Young inequalities leads to
\begin{equation}
\label{eq:liap-v}
 \frac{1}{2} \frac{\mathrm{d}}{\mathrm{d}t}  \|v\|_H^2 \leq - \frac{\alpha}{2} \|v\|_H^2 + \frac{1}{2\alpha} \|[\mathrm{d}F(\mathcal{T}_t w_1) - \mathrm{d}F( \mathcal{T}_t w_2)]\d \T_t(w_2)\|_{\mathcal{L}(H)}^2 \|h\|_H^2.
\end{equation}
Since $v(0) = 0$, we deduce from \eqref{eq:liap-v} multiplied by $\exp(\alpha t)$, \cref{eq:decay-linearized} and \cref{eq:exp-cont} that
\begin{equation}
\label{eq:est-unif-v}
 \|v(t)\|^2_H 
 \leq \frac{ K_R^2 \exp(- \alpha t)}{3\alpha^2} \|w_1 - w_2\|^2_H \|h\|^2_H, \quad \forall  t \geq 0.
\end{equation}
We infer from \cref{eq:est-Ttw1} and \cref{eq:est-unif-v} that
\begin{equation}
\label{eq:est-first-term}
\|\d F(\T_t w_1)[\d \T_t (w_1) - \d \T_t (w_2)]\|_{\mathcal{L}(H)} \leq m R K_R^2 \exp(-\alpha t/2) \|w_1 - w_2\|_H 
\end{equation}
for all $t\geq 0$, where $m$ is some constant independent of $R$; 
on the other hand, coming back to the second term of \cref{eq:decomposition}, we also have
\begin{equation}
\label{eq:est-second-term}
\| [\d F (\T_t w_1) - \d F(\T_t w_2)] \d \T_t (w_2)\|_{\mathcal{L}(H)} \leq K_R \exp(- 2\alpha t) \|w_1 - w_2\|_H.
\end{equation}
Thus, the desired local Lipschitz continuity is obtained by applying the triangular inequality to \cref{eq:decomposition} and integrating 
\cref{eq:est-first-term} and \cref{eq:est-second-term} over $(0, + \infty)$.
Furthermore, if we assume that both $F$ and $\d F$ are globally Lipschitz continuous, then for some $K_F$ we can choose $K_R = K_F$ independent of $R$ and replace \cref{eq:est-Ttw1} with $\|\d F(\T_t w_1) \|_{\mathcal{L}(H)} \leq K_F$, thereby proving global Lipschitz continuity of the differential $\d \M$. 

\textbf{Step 3: Conclusion.} At this point, it remains to check that our candidate $\M$ is a solution to \cref{eq:forwarding-func}.
It is clear that $\M(0) = 0$. Take $w_0$ in $\mathcal{D}(\mathcal{A})$ and consider the associated strong solution
 $w \triangleq t \mapsto  \mathcal{T}_t w_0 \in \mathcal{C}(\R_+, \mathcal{D}(A))\cap \mathcal{C}^1(\R_+, H)$ to \eqref{eq:original-w}.
Then, 
\begin{equation}
\label{eq:diff-M}
\frac{\mathcal{M}(\T_t w_0) - \mathcal{M}(w_0)}{t} = C \left [ \frac{1}{t}\int_0^t \mathcal{T}_s w_0 \, \mathrm{d}s \right ].
\end{equation}
As $w$ is continuous in $\mathcal{D}(A)$, the term
$ t^{-1} \int_0^t \mathcal{T}_sw_0 \, \mathrm{d}s$ converges to $w_0$ in $\D(A)$ when $t$ approaches $0$. Thus, $C$ being $A$-bounded, the difference quotient in \eqref{eq:diff-M} converges to $Cw_0$ in $Z$.
On the other hand, since $w \in \mathcal{C}^1(\R_+, H)$ and $\mathcal{M} \in \mathcal{C}^1(H, Z)$, by the chain rule, we have
\begin{equation}
\label{eq:chain-rule}
\frac{\mathrm{d}}{\mathrm{d}t} [ \mathcal{M}(\T_t w_0)] = -\mathrm{d}\mathcal{M}(\T_t w_0) \mathcal{A}(\T_t w_0), \quad \forall t \geq 0.
\end{equation}
Evaluating \eqref{eq:chain-rule} at $t = 0$ yields
$\mathrm{d}\mathcal{M}(w_0)\mathcal{A}(w_0) + Cw_0 = 0$ by uniqueness of the limit.

\section{Examples}
\label{sec:examples}
We provide three examples for which our previous results apply. In what follows, we use standard notation for real-valued Lebesgue and Sobolev spaces.

\subsection{Sine-Gordon equation}
\label{sec:sine-gordon}
Let $\xi$, $\gamma$, and $L$ be positive constants. Let $\mathcal{O}$ be a non-empty open subset of $\Omega \triangleq (0, L)$. Consider the following damped sine-Gordon equation with control acting on $\mathcal{O}$ and homogeneous Dirichlet boundary conditions:
\begin{subequations}
\label{eq:sine-gordon}
\begin{align}
&\partial_{tt} \theta + \xi \partial_t \theta - \partial_{xx} \theta + \gamma \sin(\theta) =  u(t) \mathds{1}_\mathcal{O}  & &\mbox{in}~ (0, L) \times (0, +\infty), \\
&\theta(0, t) = \theta(L, t) = 0 & &\mbox{for all}~ t\geq 0.
\end{align}
\end{subequations}
In this example inspired by \cite[Chapter IV]{temam_infinite-dimensional_2012}, \cref{eq:sine-gordon} may represent the voltage dynamics of the continuous limit case for coupled Josephson junctions, with the control $u$ being proportional to the applied current.
The uncontrolled dynamics generated by \cref{eq:sine-gordon} are well-posed on the energy space $H \triangleq H^1_0(\Omega) \times L^2(\Omega)$.
Using the state variable $w = [\theta, \partial_t \theta]$, we can recast \cref{eq:sine-gordon} into a semilinear evolution problem on $H$ as in \cref{sec:semilinear}. Letting $\D (A) \triangleq [H^2(\Omega) \times H^1_0(\Omega)]\cap H$ and  $U \triangleq \R$, we define the unbounded linear operator $A$, the input operator $B$, and the nonlinear mapping $F$ by
\begin{subequations}
\begin{align}
&A[\theta, \zeta] \triangleq [- \zeta, - \partial_{xx} \theta + \xi \zeta + \gamma \theta] & & \forall [\theta, \zeta] \in \D (A), \\
& Bu \triangleq [0,  u \mathds{1}_\mathcal{O}] & & \forall u \in U, \\
&F[\theta, \zeta] \triangleq [0, \gamma \sin (\theta) - \gamma \theta] & & \forall [\theta, \zeta] \in H.
\end{align}
\end{subequations}
As an output, consider the Neumann trace at, say, $x = 0$:
\begin{equation}
y(t) = Cw(t) = \partial_x \theta(0, t),
\end{equation}
which is modeled by an \emph{unbounded} (but $A$-bounded) operator. That $A^{-1}$ exists in $\mathcal{L}(H)$ can be proved using Riesz representation theorem in $H^1_0(\Omega)$. Besides, $F$ is in  $\mathcal{C}^1(H)$ and both $F$ and $\d F$ are globally Lipschitz continuous
, $\d F$ being given by
\begin{equation}
\d F(\theta, \zeta)[h_1, h_2] = [0, \gamma \cos(\theta) h_1 - \gamma h_1], \quad \forall [\theta, \zeta], [h_1, h_2] \in H.
\end{equation}
We equip $H$ with a scalar product
that is equivalent to the usual one:
\begin{equation}
([\theta_1, \zeta_1 ], [\theta_2, \zeta_2])_{H, \eps} \triangleq \int_\Omega \partial_x \theta_1 \partial_x \theta_2 \, \d x + \int_\Omega (\zeta_1 + \eps \theta_1 ) ( \zeta_2 + \eps \theta_2) \, \d x
\end{equation}
where $\eps \triangleq \min \{ \xi/4, \lambda_1 /(2 \xi) \}$, with $\lambda_1$ being the optimal Poincar\'e inequality constant. Then, after some computations similar to \cite[Section IV.1.2]{temam_infinite-dimensional_2012},  we obtain
\begin{equation}
(Ah + \d F(\theta, \eta)h, h)_{H, \eps} \geq \frac{\eps}{2} \|h\|_{H, \eps}^2 - \gamma \lambda_1 \|h\|^2_{H, \epsilon}, \quad \forall h \in \D(A), \ \forall [\theta, \zeta] \in H.
\end{equation}
Therefore, \cref{hyp:semilinear} is satisfied as long as $\gamma < \eps /(2\lambda_1)$. In that case, \cref{th:M-semilinear} provides a suitable solution $\M$ to \cref{eq:forwarding-func} with $\d \M$ globally Lipschitz continuous, upon which a forwarding control law can be built for the output regulation problem. Since the range of $CA^{-1}B$ is non-zero\footnote{
Consider for instance the image by $CA^{-1}B$ of $[0,\mathds{1}_I]$ where $I$ is some segment contained  in $\mathcal{O}$.
} and hence $\R$, \cref{th:global-stab} guarantees the existence of a locally exponentially stable equilibrium for the closed-loop system with small reference and disturbance. Furthermore, following the discussion subsequent to \cref{th:global-stab}, the global coercivity condition \cref{eq:unif-coercive} holds whenever $\eps /2(1 + \lambda_1) > \gamma$, in which case the theorem provides a globally asymptotically stable equilibrium.

\subsection{A pre-stabilized Wilson-Cowan equation} The following example is inspired by the study of neural fields (see for instance \cite{boscain_bio-inspired_2021}). Let $\Omega$ be a bounded domain in $\R^n$ and $\mathcal{O}$ be an open subset of $\Omega$. Given a positive gain $\alpha$, a kernel $k \in L^\infty(\Omega \times \Omega)$, and a smooth scalar nonlinearity $s$ that has bounded derivative and vanishes at $0$, consider the following non-local evolution equation:
\begin{equation}
\label{eq:wil-cow}
\partial_t w(x, t) + \alpha w(x, t) + \int_\Omega k(x,\nu) s(w(\nu, t)) \, \d \nu = \mathds{1}_\mathcal{O}(x) u(x, t) \quad \mbox{on}~ \Omega \times (0, +\infty).
\end{equation}
We look at the (vector-valued) output given by
\begin{equation}
y(t) = Cw(t) = w_{|\mathcal{O}}(t).
\end{equation}
Having set $H \triangleq L^2(\Omega)$, we define two  mappings $K$ and $F$ on $H$ by $[Kw](x) \triangleq s'(0) \int_\Omega k(x, \nu) w(\nu) \, \d \nu$ and [$F(w)](x) \triangleq \int_\Omega k(x, \nu) \{ s(w(\nu)) - s'(0)w(\nu) \} \, \d \nu$ for all $w \in H$. Then, $K$ is a bounded linear operator, and $F$ is continuously differentiable with $F$ and $\d F$ globally Lipschitz continuous.  We also let $A \triangleq \alpha \mathrm{id} + K$ and $Bu \triangleq \mathds{1}_\mathcal{O}u$ with $U \triangleq L^2(\mathcal{O})$. As an integral operator, $K$ is compact; thus, $0$ lies in the resolvent set of $A$ except for a bounded and countable set of values for $\alpha$.
Here, $A$ is continuous; therefore, we can choose $Z \triangleq \operatorname{Range} A^{-1}B$, which as a closed subspace of $H$ is a Hilbert space as well. By letting $C \triangleq \mathrm{id}$, the condition   \cref{eq:non-resonance}  is automatically satisfied. Furthermore, $Z = H$ if and only if $\mathcal{O} = \Omega$.
As for condition \cref{eq:as-stab-lin}, we have
\begin{equation}
(Ah + \d F(w)h, h)_H \geq (\alpha - M_{k, s}) \|h\|^2_H, \quad \forall w, h \in H,
\end{equation}
where $M_{k, s} \triangleq \iint_{\Omega \times \Omega} |k(x, \nu) s'(\nu)|^2 \, \d x\,  \d \nu$. \cref{hyp:semilinear} is satisfied whenever $\alpha > M_{k, s}$, while  global uniform coercivity \cref{eq:unif-coercive} holds provided that $\alpha > 2M_{k,s}$.

\subsection{A reaction-diffusion equation}
\label{sec:w3}
Let $\Omega \triangleq (0, L)$, $L >0$. Consider the following reaction-diffusion equation with homogeneous Dirichlet boundary conditions:
\begin{subequations}
\label{eq:w-dyn}
\begin{align}
&\partial_t w - \partial_{xx} w + w^3 = 0 &&\mbox{in}~ (0, L) \times (0, +\infty),\\
&w(0, t) = w(L, t) = 0 &&\mbox{for}~ t \in (0, +\infty).
\end{align}
\end{subequations}
To \cref{eq:w-dyn} we can associate
a contraction semigroup $\{\T_t\}$ on $H \triangleq L^2(\Omega)$ with generator $- \A$ satisfying \cref{hyp:nonlin-coer}. Note that unlike the previous examples, this problem does \emph{not} enter the realm of application of \cref{sec:semilinear} as the (pointwise) function $w \mapsto w^3$ cannot be modelled by a nonlinear map $F:H\to H$. Nevertheless, it can be proved that $\{\T_t\}$ is Fr\'echet differentiable \cite[Section VI.2.1]{temam_infinite-dimensional_2012}, the first variation equation around a trajectory $w$ of \cref{eq:w-dyn} being given by
\begin{subequations}
\label{eq:first-var-cube}
\begin{align}
&\partial_t v - \partial_{xx}v + 3w^2 v = 0 &&\mbox{in}~ (0, L) \times (0, +\infty),\\
&v(0, t) = v(L, t) = 0 &&\mbox{for}~ t \in (0, +\infty).
\end{align}
\end{subequations}
We start by proving that
$
w \mapsto \int_0^{+\infty} \d \T_t (w) \, \d t
$
is a locally Lipschitz continuous map of $H$ into $\L(H)$. Following the arguments given at the very end of the proof of \cref{th:M-semilinear}, this will have the consequence that for each \emph{bounded} output operator $C$, $\M$ defined by $\M(w) \triangleq -C \int_0^{+\infty} \T_t w \, \d t$ is a solution to \cref{eq:forwarding-func} with locally Lipchitz continuous differential.

 Let $w_1^0$ and $w_2^0$ be taken in $H$, with $\|w_i^0\|_H \leq R$. In what follows, approximation arguments are omitted. Denote by $w_i$ the corresponding solution $t\mapsto \T_t w_i^0$ to \cref{eq:w-dyn}. Let $h \in H$ and
$
v_i(t) \triangleq \d \T_t (w_i)h
$, $t \geq 0$.
Then, by \cref{eq:first-var-cube}, $v \triangleq v_1 - v_2$ solves
\begin{subequations}
\begin{align}
\label{eq:var-var}
&\partial_t  v - \partial_{xx} v + 3w_1^2 v_1 - 3w_2^2 v_2 = 0 &&\mbox{in}~ (0, L) \times (0, +\infty),\\
&v(0, t) = v(L, t) = 0 &&\mbox{for}~ t \in (0, +\infty),
\end{align}
\end{subequations}
with  $v(\cdot, 0) = 0$. Letting $w \triangleq w_1 - w_2$, we infer from \cref{eq:var-var} that
\begin{equation}
\label{eq:var-var-bis}
\partial_t v - \partial_{xx} v + 3w_1^2 v + 3 w\{w_1 + w_2\} v_2 = 0 \quad \mbox{in}~ (0, ) \times (0, + \infty).
\end{equation}
Let $\tau > 0$ and $Q_\tau \triangleq \Omega \times (0, \tau)$. Multiplying \cref{eq:var-var-bis} by $v$ and integrating over $Q_\tau$ yields
\begin{equation}
\label{eq:ipp-cube}
\frac{1}{2} \int_\Omega |v(\cdot, \tau)|^2 \, \d x + \iint_{Q_\tau}  | \partial_x v  |^2 + 3w_1^2 |v|^2  + 3 \{w_1 + w_2\} v_2 w v \, \d x \, \d t = 0,
\end{equation}
For a.e. $(x, t)$ in $Q_\tau$, we have
\begin{multline}
\label{eq:est-pointwise}
| \{w_1(x, t) + w_2(x,t)\} v_2(x, t) w(x, t) v(x, t)|  \\\leq K \left \{ \|w_1(t)\|_{H^1_0(\Omega)} + \|w_2( t)\|_{H^1_0(\Omega)} \right \} \|v( t)\|_{H^1_0(\Omega)} |w(x, t) v_2 (x, t)|,
\end{multline}
where $K > 0$ comes from the continuous embedding $H^1_0(\Omega) \hookrightarrow L^\infty(\Omega)$ and does not depend on $\tau$. Integrating  \cref{eq:est-pointwise} over $\Omega$, we obtain that for a.e. $t \in (0, \tau)$,
\begin{multline}
\label{eq:est-cs-cube}
\int_\Omega |\{w_1(x, t) + w_2(x, t)\} v_2(x, t) w(x, t) v(x, t)| \, \d x    \\ \leq K  \left \{ \|w_1(t)\|_{H^1_0(\Omega)} + \|w_2( t)\|_{H^1_0(\Omega)}  \right \} \|v(t) \|_{H^1_0(\Omega)} \|w(t)\|_{H}  \|v_2(t)\|_{H}.
\end{multline}
On the other hand, by the contraction properties of \cref{eq:w-dyn,eq:first-var-cube}, 
\begin{equation}
\label{eq:cont-w-var}
\|w(t)\|_{H} \leq  \|w_1^0 - w_2^0\|_{H}, \quad \|v_2(t)\|_{H} \leq \|h\|_{H}, \quad \forall t \geq 0.
\end{equation}
Therefore, plugging \cref{eq:est-cs-cube,eq:cont-w-var} into \cref{eq:ipp-cube} leads to
\begin{multline}
\label{eq:est-cs-cube-bis}
\frac{1}{2} \int_\Omega |v(\cdot, \tau)|^2 \, \d x + \iint_{Q_\tau}  | \partial_x v |^2 \, \d x \, \d t\\
\leq 3K   \|w_1^0 - w_2^0\|_{H} \|h\|_{H} \int_0^\tau  \left \{ \|w_1( t)\|_{H^1_0(\Omega)} + \|w_2(t)\|_{H^1_0(\Omega)} \right \} \|v(t)\|_{H^1_0(\Omega)} \, \d t.
\end{multline}
At this point we take advantage of additional regularity properties of \cref{eq:w-dyn}, namely
\begin{equation}
\|w_i\|_{L^2(0, +\infty; H^1_0(\Omega))}^2 = \int_0^{+ \infty} \int_\Omega |\partial_x w_i |^2 \d x \, \d t \leq \frac{1}{2} \|w_i^0\|^2_H,
\end{equation}
which follows from integrating the energy identity
\begin{equation}
\frac{1}{2} \dt{} \left \{ \int_\Omega |w_i|^2 \, \d x \right \} + \int_\Omega |\partial_x w_i|^2 \, \d x + \int_\Omega |w_i|^4 \, \d x = 0.
\end{equation}
We can then use Cauchy-Schwarz and Young inequalities in  \cref{eq:est-cs-cube-bis} to obtain
\begin{multline}
\frac{1}{2} \int_\Omega |v(\cdot, \tau)|^2 \, \d x + \iint_{Q_\tau}  | \partial_x v  |^2 \,  \d x \, \d t \\
\leq  9K^2   \|w_1^0 - w_2^0\|_{H}^2 \|h\|_{H}^2 \{\|w_1^0\|^2_{H} + \|w_2^0\|^2_{H}\} + \frac{1}{2}\iint_{Q_\tau}  | \partial_x v |^2 \,  \d x \, \d t.
\end{multline}
Recalling that $\|w_i^0\|_{L^2(\Omega)} \leq R$, we get
\begin{equation}
\frac{1}{2} \int_\Omega |v(\cdot, \tau)|^2 \, \d x  + \frac{1}{2}  \iint_{Q_\tau}  | \partial_x v |^2  \d x \, \d t \leq 18 K^2 R^2   \|w_1^0 - w_2^0\|_{H}^2 \|h\|_{H}^2.
\end{equation}
Using the Poincaré inequality $\|v\|^2_{H^1_0(\Omega)} \geq \lambda_1 \|v\|^2_H$, we obtain
\begin{equation}
\frac{1}{2} \int_\Omega |v(\cdot, \tau)|^2 \, \d x  + \frac{\lambda_1}{2} \iint_{Q_\tau} |v|^2 \, \d x \, \d t \leq {18 K^2 R^2}\|w_1^0 - w_2^0\|_{H}^2 \|h\|_{H}^2.
\end{equation}
This holds for all $\tau > 0$. Hence, an appropriate Gr\"onwall inequality yields
\begin{equation}
\label{eq:uniform-estimate}
\|v(t)\|^2_{H} \leq K' R^2 \exp(- \lambda_1 t)  \|w_1^0 - w_2^0\|_{H}^2 \|h\|_{H}^2, \quad \forall t \geq 0.
\end{equation}
\Cref{eq:uniform-estimate} is valid for all $h \in H$ and any $w_i^0$ in the ball of (arbitrary) radius $R$ in $H$. Thus, similarly as in the proof of \cref{th:M-semilinear}, our first claim that $
w \mapsto \int_0^{+\infty} \d \T_t (w) \, \d t
$
is a locally Lipschitz continuous map of $H$ into $\L(H)$ follows.

Let $B : U \to H$ and  $C : H \to Z$ be a bounded input and output operators. With our choice of $\M$, the ``non-resonance condition'' reads as $\operatorname{Range} CA^{-1}B= Z$, where $A$ is $- \partial_{xx}$ with zero Dirichlet boundary conditions. As in \cref{sec:sine-gordon}, we choose a non-empty open subset $\mathcal{O}$ of $\Omega$ and let $Bu = u \mathds{1}_\mathcal{O}$, $u \in \R$. By virtue of \cref{th:global-stab} (in the case that $\d \M$ is only locally Lipschitz continuous, which requires $d = 0$), this localized scalar control achieves local regulation of the nonlinear reaction-diffusion equation \cref{eq:w-dyn} with one-dimensional bounded output maps such as the mean-value $Cw = \int_{\Omega} w$.

\begin{rem}
    This example shows that the arguments used in \cref{sec:proof-semilinear} can be adapted beyond the semilinear context to derive existence and suitable properties of $\M$ and $\d\M$ using specific features of the system, such as additional regularity and dissipativity in higher-order norms.
\end{rem}

\section{Conclusions} We finish our article with some comments on our results.
We wish to point out two notable byproducts of our approach in terms of \emph{robustness}.

In order to take into account nonlinear behavior of actuators (e.g., saturation) in the feedback loop, one may investigate stability properties of \cref{eq:control-system}-\cref{eq:feedback-law}
when a nonlinearity $g$ is applied to the input $u$.
This is also relevant in applications where the control signal must satisfy some prescribed bound in norm (see e.g. \cite{marx:hal-02944073}).
Consider a Lipschitz continuous map $g$ that vanishes at $0$ and is
\emph{strongly} monotone in some neighborhood of $0$. Then, the local (strict) contraction property  of the $[w, \eta]$-dynamics is preserved. This leaves room for a possible adaptation of \cref{th:global-stab} in the case of saturated or \textit{a priori} bounded control.

We also believe that our framework provides tools to analyse the behavior of the closed-loop \cref{eq:control-system}-\cref{eq:feedback-law} under  certain  {time-varying} disturbances. Indeed, given $[d_0, \yr]$ as in \cref{th:global-stab}, to which we associate an equilibrium $[w^\star, z^\star]$, consider a disturbance of the form $d(t) = d_0 + d_1(t)$ with $d_1$ small in $L^2(0, +\infty; H)$. It can then be deduced from \cref{eq:local-cont} that the system \cref{eq:control-system}-\cref{eq:feedback-law} is \emph{incrementally} input-to-state stable in a neighborhood of $[w^\star, z^\star]$, allowing us to quantify the deviation from equilibrium due to the exogenous signal $d_1$ in terms of its $L^2$-norm.

Finally, putting aside the problem of output regulation and following \cite{marx:hal-02944073, marx:hal-03417238}, we might be interested in stabilizing the cascade composed of the (actuated) $w$-subsystem and a more general $z$-subsystem governed by $\d z/ \d t = Sz + Cw$, where $S$ is a 
skew-adjoint
operator on $Z$.
This would require to investigate a nonlinear Sylvester equation of the form
\begin{equation}
\label{eq:nonlin-sylvester}
\d\M(w)\A(w) + S \M(w) + Cw = 0, \quad \forall w\in \D (\A).
\end{equation}
Under the condition that a sufficiently regular solution $\M$ to \cref{eq:nonlin-sylvester} exists, the Lyapunov analysis performed in \cref{sec:forwarding} remains valid, which is a good starting point for analysing stability of the new closed-loop with control law given by \cref{eq:feedback-law}.

\section*{Acknowledgments}
The authors would like to thank Vincent Andrieu, Daniele Astolfi and Christophe Prieur for many fruitful discussions.

\bibliographystyle{alpha}
\bibliography{forwarding-final.bib}

\begin{thebibliography}{TJAMDSX19}

\bibitem[AP17]{astolfi-praly}
Daniele Astolfi and Laurent Praly.
\newblock Integral action in output feedback for multi-input multi-output
  nonlinear systems.
\newblock {\em IEEE Transactions on Automatic Control}, 62(4):1559--1574, 2017.

\bibitem[BAPH13]{benachour2013forwarding}
Mohamed~Sofiane Benachour, Vincent Andrieu, Laurent Praly, and Hassan Hammouri.
\newblock Forwarding design with prescribed local behavior.
\newblock {\em IEEE Transactions on Automatic Control}, 58(12):3011--3023,
  2013.

\bibitem[BLGS00]{bymes2000output}
Christopher~I. Bymes, Istv{\'a}n~G. Lauk{\'o}, David~S. Gilliam, and Victor~I.
  Shubov.
\newblock Output regulation for linear distributed parameter systems.
\newblock {\em IEEE Transactions on Automatic Control}, 45(12):2236--2252,
  2000.

\bibitem[BMA22]{balogoun:hal-03292801}
Ismaila Balogoun, Swann Marx, and Daniele Astolfi.
\newblock {ISS} {L}yapunov strictification via observer design and integral
  action control for a {K}orteweg-de {V}ries equation.
\newblock Accepted for publication in \textit{SIAM Journal on Control and
  Optimization}, 2022.

\bibitem[BPST21]{boscain_bio-inspired_2021}
Ugo Boscain, Dario Prandi, Ludovic Sacchelli, and Giuseppina Turco.
\newblock A bio-inspired geometric model for sound reconstruction.
\newblock {\em The Journal of Mathematical Neuroscience}, 11(1):1--18, 2021.

\bibitem[CEL02]{chueshov_attractor_2002}
Igor Chueshov, Matthias Eller, and Irena Lasiecka.
\newblock On the attractor for a semilinear wave equation with critical
  exponent and nonlinear boundary dissipation.
\newblock {\em Communications in Partial Differential Equations},
  27(9-10):1901--1951, January 2002.

\bibitem[Chu02]{chueshov_introduction_2002}
Igor Chueshov.
\newblock {\em Introduction to the theory of infinite-dimensional dissipative
  systems}.
\newblock University lectures in contemporary mathematics. Acta Scientific
  Publ. House, Kharkiv, 2002.

\bibitem[CP69]{crandall_semi-groups_1969}
Michael~G. Crandall and Amnon Pazy.
\newblock Semi-groups of nonlinear contractions and dissipative sets.
\newblock {\em Journal of functional analysis}, 3(3):376--418, 1969.

\bibitem[CZ20]{CurZwa20book}
Ruth Curtain and Hans Zwart.
\newblock {\em Introduction to infinite-dimensional systems theory}, volume~71
  of {\em Texts in Applied Mathematics}.
\newblock Springer, New York, 2020.
\newblock A state-space approach.

\bibitem[Dav75]{1101104}
Edward Davison.
\newblock A generalization of the output control of linear multivariable
  systems with unmeasurable arbitrary disturbances.
\newblock {\em IEEE Transactions on Automatic Control}, 20(6):788--792, 1975.

\bibitem[Dav76]{Davison}
Edward Davison.
\newblock The robust control of a servomechanism problem for linear
  time-invariant multivariable systems.
\newblock {\em IEEE Transactions on Automatic Control}, 21(1):25--34, 1976.

\bibitem[DS73]{dafermos_asymptotic_1973}
Constantine~M. Dafermos and Marshall Slemrod.
\newblock Asymptotic behavior of nonlinear contraction semigroups.
\newblock {\em Journal of Functional Analysis}, 13(1):97--106, May 1973.

\bibitem[FW75]{francis1975internal}
Bruce~A. Francis and William~M. Wonham.
\newblock The internal model principle for linear multivariable regulators.
\newblock {\em Applied mathematics and optimization}, 2(2):170--194, 1975.

\bibitem[GAAM21]{mattia-forwarding}
Mattia Giaccagli, Daniele Astolfi, Vincent Andrieu, and Lorenzo Marconi.
\newblock Sufficient conditions for global integral action via incremental
  forwarding for input-affine nonlinear systems.
\newblock {\em IEEE Transactions on Automatic Control}, 67(12):6537--6551,
  2021.

\bibitem[GZ22]{guo2022output}
Bao-Zhu Guo and Ren-Xi Zhao.
\newblock Output regulation for a heat equation with unknown exosystem.
\newblock {\em Automatica}, 138:110159, 2022.

\bibitem[GZK18]{guo2018adaptive}
Wei Guo, Hua-Cheng Zhou, and Miroslav Krstic.
\newblock Adaptive error feedback regulation problem for {1D} wave equation.
\newblock {\em International Journal of Robust and Nonlinear Control},
  28(15):4309--4329, 2018.

\bibitem[IMS03]{isidori2003robust}
Alberto Isidori, Lorenzo Marconi, and Andrea Serrani.
\newblock {\em Robust autonomous guidance: an internal model approach}.
\newblock Springer Science \& Business Media, 2003.

\bibitem[KA05]{kaliora2005stabilization}
Georgia Kaliora and Alessandro Astolfi.
\newblock On the stabilization of feedforward systems with bounded control.
\newblock {\em Systems $\&$ Control Letters}, 54(3):263--270, 2005.

\bibitem[Kom69]{komura_differentiability_1969}
Yukio Komura.
\newblock Differentiability of nonlinear semigroups.
\newblock {\em Journal of the Mathematical Society of Japan}, 21(3):375--402,
  1969.

\bibitem[LPT21]{lhachemi2021}
Hugo Lhachemi, Christophe Prieur, and Emmanuel Tr{\'e}lat.
\newblock {PI} regulation of a reaction--diffusion equation with delayed
  boundary control.
\newblock {\em IEEE Transactions on Automatic Control}, 66(4):1573--1587, 2021.

\bibitem[LR00]{logemann2000time}
Hartmut Logemann and Eugene~P. Ryan.
\newblock Time-varying and adaptive integral control of infinite-dimensional
  regular linear systems with input nonlinearities.
\newblock {\em SIAM Journal on Control and Optimization}, 38(4):1120--1144,
  2000.

\bibitem[LS98]{lohmiller-contraction}
Winfried {Lohmiller} and Jean-Jacques~E. {Slotine}.
\newblock {On contraction analysis for non-linear systems}.
\newblock {\em {Automatica}}, 34(6):683--696, 1998.

\bibitem[MAA22]{marx:hal-03417238}
Swann Marx, Daniele Astolfi, and Vincent Andrieu.
\newblock Forwarding-{Lyapunov} design for the stabilization of coupled {ODEs}
  and exponentially stable {PDEs}.
\newblock In {\em 2022 European Control Conference (ECC)}, pages 339--344,
  2022.

\bibitem[MBA21]{marx:hal-02944073}
Swann Marx, Lucas Brivadis, and Daniele Astolfi.
\newblock {Forwarding techniques for the global stabilization of dissipative
  infinite-dimensional systems coupled with an ODE}.
\newblock {\em Mathematics of Control, Signals, and Systems}, 33:755--774,
  2021.

\bibitem[MP96]{praly-mazenc-forwarding}
Frederic Mazenc and Laurent Praly.
\newblock Adding integrations, saturated controls, and stabilization for
  feedforward systems.
\newblock {\em IEEE Transactions on Automatic Control}, 41(11):1559--1578,
  1996.

\bibitem[NB16]{natarajan2016approximate}
Vivek Natarajan and Joseph Bentsman.
\newblock Approximate local output regulation for nonlinear distributed
  parameter systems.
\newblock {\em Mathematics of Control, Signals, and Systems}, 28(3):1--44,
  2016.

\bibitem[Pau15]{paunonen2015controller}
Lassi Paunonen.
\newblock Controller design for robust output regulation of regular linear
  systems.
\newblock {\em IEEE Transactions on Automatic Control}, 61(10):2974--2986,
  2015.

\bibitem[Pau19]{doi:10.1137/17M1136407}
Lassi Paunonen.
\newblock Stability and robust regulation of passive linear systems.
\newblock {\em SIAM Journal on Control and Optimization}, 57(6):3827--3856,
  2019.

\bibitem[Paz12]{pazy_semigroups_2012}
Amnon Pazy.
\newblock {\em Semigroups of linear operators and applications to partial
  differential equations}, volume~44.
\newblock Springer Science \& Business Media, 2012.

\bibitem[Poh82]{Pohjolainen}
Seppo Pohjolainen.
\newblock Robust multivariable {PI}-controller for infinite dimensional
  systems.
\newblock {\em IEEE Transactions on Automatic Control}, 27(1):17--30, 1982.

\bibitem[Poh85]{POHJOLAINEN1985622}
Seppo Pohjolainen.
\newblock Robust controller for systems with exponentially stable strongly
  continuous semigroups.
\newblock {\em Journal of Mathematical Analysis and Applications},
  111(2):622--636, 1985.

\bibitem[PP10a]{paunonen2010internal}
Lassi Paunonen and Seppo Pohjolainen.
\newblock Internal model theory for distributed parameter systems.
\newblock {\em SIAM Journal on Control and Optimization}, 48(7):4753--4775,
  2010.

\bibitem[PP10b]{poulain2010robust}
Fran{\c{c}}ois Poulain and Laurent Praly.
\newblock Robust asymptotic stabilization of nonlinear systems by state
  feedback.
\newblock {\em IFAC Proceedings Volumes}, 43(14):653--658, 2010.

\bibitem[RW03]{rebarber2003internal}
Richard Rebarber and George Weiss.
\newblock Internal model based tracking and disturbance rejection for stable
  well-posed systems.
\newblock {\em Automatica}, 39(9):1555--1569, 2003.

\bibitem[Sho13]{showalter_monotone_2013}
Ralph~E. Showalter.
\newblock {\em Monotone operators in {Banach} space and nonlinear partial
  differential equations}, volume~49.
\newblock American Mathematical Society, 2013.

\bibitem[Tem12]{temam_infinite-dimensional_2012}
Roger Temam.
\newblock {\em Infinite-dimensional dynamical systems in mechanics and
  physics}, volume~68.
\newblock Springer Science \& Business Media, 2012.

\bibitem[TJAMDSX19]{terrand2019adding}
Alexandre Terrand-Jeanne, Vincent Andrieu, Valérie Martins Dos-Santos, and
  Cheng-Zhong Xu.
\newblock Adding integral action for open-loop exponentially stable semigroups
  and application to boundary control of {PDE} systems.
\newblock {\em IEEE Transactions on Automatic Control}, 65(11):4481--4492,
  2019.

\bibitem[TW09]{tucsnak2009observation}
Marius Tucsnak and George Weiss.
\newblock {\em Observation and control for operator semigroups}.
\newblock Springer Science \& Business Media, 2009.

\end{thebibliography}

\end{document}